\documentclass[12pt]{amsart}

\usepackage{amssymb}

\usepackage{enumerate}

\usepackage{graphicx}

\makeatletter
\@namedef{subjclassname@2010}{%
  \textup{2010} Mathematics Subject Classification}
\makeatother

\newtheorem{thm}{Theorem}[section]

\newtheorem{lem}[thm]{Lemma}

\newtheorem{example}[thm]{Example}
\newtheorem{definition}[thm]{Definition}

\newcommand{\bbQ}{\mathbb Q}

\theoremstyle{definition}

\numberwithin{equation}{section}

\begin{document}

%%%%% To ease editing, for IMPAN journals add:

\baselineskip=17pt

%%%%%%%%%%%

%% In the running head, replace first names by initials 
%% and give an abbreviation of the title.

\title[On the $x$--Coordinates of Pell Equations ...]{On the $x$--Coordinates of Pell Equations which are $k$--Generalized Fibonacci Numbers}

\author[M. Ddamulira]{Mahadi Ddamulira}
\address{Institute of Analysis and Number Theory \newline \indent Graz University of Technology \newline \indent Kopernikusgasse 24/II \newline \indent A-8010 Graz, Austria}
\email{mddamulira@tugraz.at}

\author[F. Luca]{Florian Luca}
\address{School of Mathematics \newline \indent University of the Witwatersrand \newline \indent Private Bag X3 \\ WITS 2050 \newline \indent Johannesberg, South Africa}
\address{Max Planck Institute for Mathematics,\newline 
         \indent Vivatsgasse 7\newline\indent  53111 Bonn, Germany}
         
\address{Department of Mathematics\newline \indent Faculty of Sciences \newline \indent University of Ostrava\newline
           \indent 30 dubna 22 \newline \indent 701 03 Ostrava 1, Czech Republic}

\email{Florian.Luca\char'100wits.ac.za}
\date{}

\begin{abstract}
For an integer $k\geq 2$, let $\{F^{(k)}_{n}\}_{n\geqslant 2-k}$ be the $ k$--generalized Fibonacci sequence which starts with $0, \ldots, 0,1$ (a total of $k$ terms) and for which each term afterwards is the sum of the $k$ preceding terms. In this paper, for an integer $d\geq 2$ which is square free, we show that there is at most one value of the positive integer $x$ participating in the Pell equation $x^{2}-dy^{2}
=\pm 1$ which is a $k$--generalized Fibonacci number, with a couple of parametric exceptions which we completely characterise. This paper extends previous work from \cite{Luca15}  for the case $k=2$ and \cite{Luca16} for the case $k=3$.
\end{abstract}



\maketitle

\section{Introduction}
Let $d\geq 2$ be a positive integer which is not a perfect square. It is well known that the Pell equation
\begin{eqnarray}
x^{2}-dy^{2}=\pm 1 \label{Pelleqn}
\end{eqnarray}
has infinitely many positive integer solutions $(x,y)$. By putting $(x_1, y_1)$ for the smallest such solution, all solutions are of the form $ (x_n, y_n) $ for some positive integer $n$, where
\begin{eqnarray}
x_n+y_n\sqrt{d} = (x_1+y_1\sqrt{d})^n\qquad {\text{\rm for~all}} \quad n\ge 1.\label{Pellsoln}
\end{eqnarray}
Recently, Luca and Togb\'e \cite{Luca15} considered the Diophantine equation
\begin{equation}\label{eq:Luca1} 
x_{n} =F_{m},
\end{equation}
where $\{F_m\}_{m \geqslant 0}$ is the sequence of Fibonacci numbers given
by $F_0 = 0$, $F_1 = 1$ and $F_{m+2} = F_{m+1} + F_m$ for all $m \geqslant 0$. They proved that  equation \eqref{eq:Luca1} has at most one solution $(n,m)$ in positive integers except for $d=2$, in which case equation \eqref{eq:Luca1} has the three solutions $(n,m)=(1,1), (1,2), (2,4)$.

Luca, Montejano, Szalay and Togb\'e \cite{Luca16} considered the Diophantine equation
\begin{equation}\label{eq:Luca2} 
x_{n}=T_m, 
\end{equation}
where $\{T_m\}_{m \geqslant 0}$ is the sequence of Tribonacci numbers given
by $T_0 = 0$, $T_1 = 1$, $T_{2} =1 $ and $T_{m+3} = T_{m+2} + T_{m+1}+T_{m}$ for all $m \geqslant 0$. They proved that equation \eqref{eq:Luca2} has at most one solution $(n,m)$ in positive integers 
for all $d$ except for $d=2$ when equation \eqref{eq:Luca2} has the three solutions $(n,m)=(1,1), (1,2), (3,5)$  and when $d=3$ case in which equation \eqref{eq:Luca2} has the two solutions $(n,m)=(1,3), (2,5)$.

The purpose of this paper is to generalize the previous results. Let $k \geqslant 2$ be an integer. We consider a generalization of Fibonacci sequence called the $k$--generalized Fibonacci sequence $\lbrace F_m^{(k)} \rbrace_{m\geqslant 2-k}$ defined as
\begin{equation}
F_m^{(k)} = F_{m-1}^{(k)} + F_{m-2}^{(k)} + \cdots + F_{m-k}^{(k)},\label{fibb1}
\end{equation}
with the initial conditions
\[F_{-(k-2)}^{(k)} = F_{-(k-3)}^{(k)} = \cdots = F_{0}^{(k)} = 0 \quad {\text {\rm  and  }}\quad F_{1}^{(k)} = 1.\]
We call $F_{m}^{(k)}$ the $m$th $k$--generalized Fibonacci number. Note that when $k=2$, it  coincides with the Fibonacci numbers and when $k=3$ it is the Tribonacci number.

The first $k+1$ nonzero terms in $F_{m}^{(k)}$ are powers of $2$, namely
\begin{equation*}
F_{1}^{(k)}=1,\quad F_{2}^{(k)}=1,\quad F_{3}^{(k)}=2, \quad F_{4}^{(k)}=4,\ldots, F_{k+1}^{(k)}=2^{k-1}.
\end{equation*}
Furthermore, the next term is $F_{k+2}^{(k)}=2^{k}-1$.
Thus, we have that
\begin{equation}\label{Fibbo111}
F_{m}^{(k)} = 2^{m-2} \quad {\text{\rm holds for all}}\quad 2\leq m\leq k+1.
\end{equation}
We also observe that the recursion \eqref{fibb1} implies the three--term recursion
$$
F_{m}^{(k)} =2F_{m-1}^{(k)} - F_{m-k-1}^{(k)} \quad {\text{\rm for all}}\quad m\geq 3,\label{fibb2}
$$
which can be used to prove by induction on $m$ that $F_m^{(k)}<2^{m-2}$ for all $m\geq k+2$ (see also \cite{BBL17}, Lemma $2$).

\section{Main Result}

In this paper, we show that there is at most one value of the positive integer $x$ participating in \eqref{Pelleqn} which is a $ k$--generalized Fibonacci number, with a couple of parametric 
exceptions that we completely characterise. This can be interpreted as solving the system of equations
\begin{align}\label{Problem}
 x_{n_1}=F_{m_1}^{(k)},\qquad x_{n_2}=F_{m_2}^{(k)},
\end{align}
with $n_2>n_1\geq 1$, $m_2>m_1\geq 2$ and $k\ge 2$. The fact that $F_1^{(k)}=F_2^{(k)}=1$, allows us to assume that $m\ge 2$. That is, if $F_m^{(k)}=1$ for some positive integer $m$, then we will assume that $m=2$. As we already mentioned, the cases $k=2$ and $ k=3 $ have been solved completely by Luca and Togb\'e \cite{Luca15} and Luca, Montejano, Szalay and Togb\'e \cite{Luca16}, respectively. 
So, we focus on the case  $k\geqslant 4$.

\medskip

We put $\epsilon:=x_1^2-dy_1^2$. Note that $dy_1^2=x_1^2-\epsilon$, so the pair $(x_1,\epsilon)$ determines $d,~y_1$. Our main result is the following: 

\begin{thm}\label{Main}
Let $k\geq 4$ be a fixed integer. Let $d\geq 2$ be a square-free integer. Assume that
\begin{equation}
\label{eq:double}
x_{n_{1}}=F_{m_1}^{(k)},\qquad {\text{and}}\quad x_{n_{2}}=F_{m_2}^{(k)}
\end{equation}
for positive integers $m_2>m_1\ge 2$ and $n_2>n_1\ge 1$, where $x_n$ is the $x$--coordinate of the $n$th solution of the Pell equation \eqref{Pelleqn}. Then, either:
\begin{itemize}
\item[(i)] $n_1=1,~n_2=2$, $m_1=(k+3)/2$, $m_2=k+2$ and $\epsilon=1$; or
\item[(ii)] $n_1=1,~n_2=3,~k=3\times 2^{a+1}+3a-5,~m_1=3\times 2^{a}+a-1,~m_2=9\times 2^{a}+3a-5$ for some positive integer $a$ and $\epsilon=1$.
\end{itemize}
\end{thm}

\section{Preliminary Results}

Here, we recall some of the facts and properties of the $k$-generalized Fibonacci sequence and solutions to Pell equations which will be used later in this paper.

\subsection{Notations and terminology from algebraic number theory} 

We begin by recalling some basic notions from algebraic number theory.

Let $\eta$ be an algebraic number of degree $d$ with minimal primitive polynomial over the integers
$$
a_0x^{d}+ a_1x^{d-1}+\cdots+a_d = a_0\prod_{i=1}^{d}(x-\eta^{(i)}),
$$
where the leading coefficient $a_0$ is positive and the $\eta^{(i)}$'s are the conjugates of $\eta$. Then the \textit{logarithmic height} of $\eta$ is given by
$$ 
h(\eta):=\dfrac{1}{d}\left( \log a_0 + \sum_{i=1}^{d}\log\left(\max\{|\eta^{(i)}|, 1\}\right)\right).
$$
In particular, if $\eta=p/q$ is a rational number with $\gcd (p,q)=1$ and $q>0$, then $h(\eta)=\log\max\{|p|, q\}$. The following are some of the properties of the logarithmic height function $h(\cdot)$, which will be used in the next sections of this paper without reference:
\begin{eqnarray}
h(\eta\pm \gamma) &\leq& h(\eta) +h(\gamma) +\log 2,\nonumber\\
h(\eta\gamma^{\pm 1})&\leq & h(\eta) + h(\gamma),\\
h(\eta^{s}) &=& |s|h(\eta) \qquad (s\in\mathbb{Z}). \nonumber
\end{eqnarray}

\subsection{$k$-generalized Fibonacci numbers}

It is known that the characteristic polynomial of the $k$--generalized Fibonacci numbers $F^{(k)}:=(F_m^{(k)})_{m\geq 2-k}$, namely
$$
\Psi_k(x) := x^k - x^{k-1} - \cdots - x - 1,
$$
is irreducible over $\bbQ[x]$ and has just one root outside the unit circle. Let $\alpha := \alpha(k)$ denote that single root, which is located between $2\left(1-2^{-k} \right)$ and $2$ (see \cite{Dresden2014}). To simplify notation, in our application we shall omit the dependence on $k$ of $\alpha$. We shall use $\alpha^{(1)}, \dotso, \alpha^{(k)}$ for all roots of $\Psi_k(x)$ with the convention that $\alpha^{(1)} := \alpha$.

We now consider for an integer $ k\geq 2 $, the function
\begin{eqnarray}\label{fun12}
f_{k}(z) = \dfrac{z-1}{2+(k+1)(z-2)} \qquad {\text{for}}\qquad z \in \mathbb{C}.
\end{eqnarray}
With this notation, Dresden and Du presented in  \cite{Dresden2014} the following ``Binet--like" formula for the terms of $F^{(k)}$:
\begin{eqnarray} \label{Binet}
F_m^{(k)} = \sum_{i=1}^{k} f_{k}(\alpha^{(i)}) {\alpha^{(i)}}^{m-1}.
\end{eqnarray}
It was proved in \cite{Dresden2014} that the contribution of the roots which are inside the unit circle to the formula (\ref{Binet}) is very small, namely that the approximation
\begin{equation} \label{approxgap}
\left| F_m^{(k)} - f_{k}(\alpha)\alpha^{m-1} \right| < \dfrac{1 }{2} \quad \mbox{holds~ for~ all~ } m \geqslant 2 - k.
\end{equation}
It was proved by Bravo and Luca in \cite{BBL17} that
\begin{eqnarray}\label{Fib12}
\alpha^{m-2} \leq F_{m}^{(k)} \leq \alpha^{m-1}\qquad {\text{\rm holds for all}}\qquad m\geq 1\quad {\text{\rm and}}\quad k\geq 2.
\end{eqnarray}
The observations from the expressions \eqref{Binet} to \eqref{Fib12} lead us to call $ \alpha$ the \textit{dominant root} of $ F^{(k)} $.

Before we conclude this section, we present some useful lemmas that will be used in the next sections on this paper. The following lemma was proved by Bravo and Luca in \cite{BBL17}.
\begin{lem}[Bravo, Luca]\label{fala5}
Let $k\geq 2$, $\alpha$ be the dominant root of $\{F^{(k)}_m\}_{m\ge 2-k}$, and consider the function $f_{k}(z)$ defined in \eqref{fun12}. Then:
\begin{itemize}
\item[(i)]\label{kat1} The inequalities
$$
\dfrac{1}{2}< f_{k}(\alpha)< \dfrac{3}{4}\qquad \text{and}\qquad |f_{k}(\alpha^{(i)})|<1, \qquad  2\leq i\leq k
$$
hold. In particular, the number $f_{k}(\alpha)$ is not an algebraic integer.
\item[(ii)]\label{kat2}The logarithmic height of $f_k(\alpha)$ satisfies $h(f_{k}(\alpha))< 3\log k$.
\end{itemize}
\end{lem}

Next, we recall the following result due to Cooper and Howard \cite{CooperHoward}.

\begin{lem}[Cooper, Howard]
\label{teoHoward}
For $k\geq 2$ and $m\geq k+2$,
\[
F_m^{(k)}=2^{m-2}+\sum_{j=1}^{\lfloor \frac{m+k}{k+1} \rfloor-1}C_{m,j} \,2^{m-(k+1)j-2},
\]
where
\[
C_{m,j}=(-1)^j \left[ \binom{m-jk}{j}-\binom{m-jk-2}{j-2}  \right].
\]
\end{lem}
In the above, we have denoted by $\lfloor x \rfloor$ the greatest integer less than or equal to $x$ and used the convention that ${\displaystyle{\binom{a}{b}=0}}$ if either $a<b$ or if one of $a$ or $b$ is negative. 

Before going further, let us see some particular cases of Lemma \ref{teoHoward}.

\begin{example}
\label{exa:1}
\begin{itemize}
\item[(i)] Assume that $m\in [2,k+1]$. Then $1< \frac{m+k}{k+1}<2$, so $\lfloor \frac{m+k}{k+1}\rfloor=1$. In this case,
$$
F_m^{(k)}=2^{m-2},
$$
a fact which we already knew. 
\item[(ii)] Assume that $m\in [k+2,2k+2]$. Then $2\le \frac{m+k}{k+1}<3$, so $\lfloor \frac{m+k}{k+1}\rfloor=2$. In this case,
\begin{eqnarray*}
F_m^{(k)} & = & 2^{m-2}+C_{m,1} 2^{m-(k+1)-2}\\
& = & 2^{m-2}-\left(\binom{m-k}{1}-\binom{m-k-2}{-1}\right)2^{m-k-3}\\
& = & 2^{m-2}-(m-k)2^{m-k-3}.
\end{eqnarray*}
\item[(iii)] Assume that $m\in [2k+3,3k+3]$. Then $3\le \frac{m+k}{k+1}<4$, so $\lfloor\frac{m+k}{k+1}\rfloor=3$. In this case,
\begin{eqnarray*}
F_m^{(k)} & = & 2^{m-2}+C_{m,1} 2^{m-(k+1)-2}+C_{m,2}2^{m-2(k+1)-2}\\
& = & 2^{m-2}-(m-k)2^{m-k-3}+\left(\binom{m-2k}{2}-\binom{m-2k-2}{0}\right)2^{m-2k-4}\\
& = & 2^{m-2}-(m-k)2^{m-k-3}+\left(\frac{(m-2k)(m-2k-1)}{2}-1\right)2^{m-2k-4}\\
& = & 2^{m-2}-(m-k)2^{m-k-3}+(m-2k+1)(m-2k-2)2^{m-2k-5}.
\end{eqnarray*}
\end{itemize}
\end{example} 

G{\'o}mez and Luca in \cite{Gomez} derived from the Cooper and Howard's formula the following asymptotic expansion of $F_m^{(k)}$ valid when $2\le m<2^k$. 

\begin{lem}[G\'omez, Luca]\label{Gomez}
If $ m<2^{k} $, then the following estimate holds:
\begin{eqnarray}\label{Cooper2}
F_{m}^{(k)}&=&2^{m-2}\left(1+\delta_1(m)\dfrac{k-m}{2^{k+1}}+\delta_2(m)\dfrac{f(k,m)}{2^{2k+2}}+\eta(k,m)\right),
\end{eqnarray}
where $ f(k,m):=\frac{1}{2}(z-1)(z+2);~ z=2k-m$, $\eta:=\eta(k,m)$ is a real number satisfying
\begin{eqnarray*}
|\eta| < \dfrac{4m^{3}}{2^{3k+3}},
\end{eqnarray*}
and $\delta_i(m)$ is the characteristic function of the set $\{m>i(k+1)\}$ for $i=1,2$. 
\end{lem}

\subsection{Linear forms in logarithms and continued fractions}

In order to prove our main result Theorem \ref{Main}, we need to use several times a Baker--type lower bound for a nonzero linear form in logarithms of algebraic numbers. There are many such 
in the literature  like that of Baker and W{\"u}stholz from \cite{bawu07}.  We use the one of Matveev from \cite{MatveevII}.  
Matveev \cite{MatveevII} proved the following theorem, which is one of our main tools in this paper.

\begin{thm}[Matveev]\label{Matveev11} Let $\gamma_1,\ldots,\gamma_t$ be positive real algebraic numbers in a real algebraic number field 
$\mathbb{K}$ of degree $D$, $b_1,\ldots,b_t$ be nonzero integers, and assume that
\begin{equation}
\label{eq:Lambda}
\Lambda:=\gamma_1^{b_1}\cdots\gamma_t^{b_t} - 1,
\end{equation}
is nonzero. Then
$$
\log |\Lambda| > -1.4\times 30^{t+3}\times t^{4.5}\times D^{2}(1+\log D)(1+\log B)A_1\cdots A_t,
$$
where
$$
B\geq\max\{|b_1|, \ldots, |b_t|\},
$$
and
$$A
_i \geq \max\{Dh(\gamma_i), |\log\gamma_i|, 0.16\},\qquad {\text{for all}}\qquad i=1,\ldots,t.
$$
\end{thm}
When $t=2$ and $\gamma_1, \gamma_2$ are positive and multiplicatively independent, we can use a result of Laurent, Mignotte and Nesterenko \cite{Laurent:1995}. Namely, let in this case $B_1, ~B_2$ be real numbers larger than $1$ such that
\begin{eqnarray*}
\log B_i\geq \max\left\{h(\gamma_i), \dfrac{|\log\gamma_i|}{D}, \dfrac{1}{D}\right\},\qquad {\text{\rm for}}\quad  i=1,2,
\end{eqnarray*}
and put
\begin{eqnarray*}
b^{\prime}:=\dfrac{|b_{1}|}{D\log B_2}+\dfrac{|b_2|}{D\log B_1}.
\end{eqnarray*}
Put
\begin{equation}
\label{eq:Gamma}
\Gamma:= b_1\log\gamma_1+b_2\log\gamma_2.
\end{equation}
We note that $\Gamma\neq0$ because $\gamma_{1} $and $\gamma_{2}$ are multiplicatively independent. The following result is Corollary $ 2 $ in \cite{Laurent:1995}.
\begin{thm}[Laurent, Mignotte, Nesterenko]\label{Matveev12}
With the above notations, assuming that $ \eta_{1}, \eta_{2}$ are positive and multiplicatively independent, then
\begin{eqnarray}
\log |\Gamma|> -24.34D^4\left(\max\left\{\log b^{\prime}+0.14, \dfrac{21}{D}, \dfrac{1}{2}\right\}\right)^{2}\log B_1\log B_2.
\end{eqnarray}
\end{thm}
Note that with $\Gamma$ given by \eqref{eq:Gamma}, we have $e^{\Gamma}-1=\Lambda$, where $\Lambda$ is given by \eqref{eq:Lambda} in case $t=2$, which explains the connection between 
Theorems \ref{Matveev11} and \ref{Matveev12}.

During the course of our calculations, we get some upper bounds on our variables which are too large, thus we need to reduce them. To do so, we use some results from the theory of continued fractions. Specifically, for a nonhomogeneous linear form in two integer variables, we use a slight variation of a result due to Dujella and Peth{\H o} \cite{dujella98}, which  itself is a generalization of a result of Baker and Davenport \cite{BD69}.

For a real number $X$, we write  $||X||:= \min\{|X-n|: n\in\mathbb{Z}\}$ for the distance from $X$ to the nearest integer.
\begin{lem}[Dujella, Peth\H o]\label{Dujjella}
Let $M$ be a positive integer, $p/q$ be a convergent of the continued fraction of the irrational number $\tau$ such that $q>6M$, and  $A,B,\mu$ be some real numbers with $A>0$ and $B>1$. Let further 
$\varepsilon: = ||\mu q||-M||\tau q||$. If $ \varepsilon > 0 $, then there is no solution to the inequality
$$
0<|u\tau-v+\mu|<AB^{-w},
$$
in positive integers $u,v$ and $w$ with
$$ 
u\le M \quad {\text{and}}\quad w\ge \dfrac{\log(Aq/\varepsilon)}{\log B}.
$$
\end{lem}

The above lemma cannot be applied when $\mu=0$ (since then $\varepsilon<0$). In this case, we use the following criterion of Legendre.

\begin{lem}[Legendre]
\label{lem:legendre}
Let $\tau$ be real number and $x,y$ integers such that
\begin{equation}
\label{eq:continuedfraction}
\left|\tau-\frac{x}{y}\right|<\frac{1}{2y^2}.
\end{equation}
Then $x/y=p_k/q_k$ is a convergent of $\tau$. Furthermore, 
\begin{equation}
\label{eq:continuedfraction1}
\left|\tau-\frac{x}{y}\right|\ge \frac{1}{(a_{k+1}+2)y^2}.
\end{equation}
\end{lem}

Finally, the following lemma is also useful. It is Lemma 7 in \cite{guzmanluca}. 

\begin{lem}[G\'uzman, Luca]
\label{gl}
If $m\geqslant 1$, $T>(4m^2)^m$  and $T>x/(\log x)^m$, then
$$
x<2^mT(\log T)^m.
$$
\end{lem}

\subsection{Pell equations and Dickson polynomials}
\label{subs:Pell}

Let $d\ge 2$ be squarefree. We put $\delta:=x_1+{\sqrt{x_1^2-\epsilon}}$ for the minimal positive integer $x_1$ such that
$$
x_1^2-dy_1^2=\epsilon,\qquad \epsilon\in \{\pm 1\}
$$
for some positive integer $y_1$. Then,
$$
x_n+{\sqrt{d}}y_n=\delta^n\qquad {\text{\rm and}}\qquad x_n-{\sqrt{d}{y_n}}=\eta^n,\qquad {\text{\rm where}}\qquad \eta:=\epsilon \delta^{-1}.
$$
From the above, we get
\begin{equation}
\label{eq:Pellsol1}
2x_n=\delta^n+(\epsilon \delta^{-1})^n\qquad {\text{\rm for~all}}\qquad n\ge 1.
\end{equation}
There is a formula expressing $2x_{n}$ in terms of $2x_1$ by means of the Dickson polynomial $D_{n}(2x_1,\epsilon)$, where 
$$
D_n(x,\nu)=\sum_{i=0}^{\lfloor n/2\rfloor} \frac{n}{n-i} \binom{n-i}{i}(-\nu)^i x^{n-2i}.
$$
These polynomials appear naturally in many number theory problems and results, most notably in a result of Bilu and Tichy \cite{BiluTichy} concerning 
polynomials $f(X), g(X)\in {\mathbb Z}[X]$ such that the Diophantine equation $f(x)=g(y)$ has infinitely many integer solutions $(x,y)$. 

\begin{example}
\label{exa:2}
\begin{itemize}
\item[(i)] $n=2$. We have
$$
2x_2=\sum_{i=0}^1 \frac{2}{2-i}\binom{2-i}{i} (-\epsilon)^i (2x_1)^{2-2i}=4x_1^2-2\epsilon,\qquad {\text{so}}\qquad x_2=2x_1^2-\epsilon.
$$
\item[(ii)] $n=3$. We have
$$
2x_3=\sum_{i=0}^1 \frac{3}{3-i}\binom{3-i}{i} (-\epsilon)^i (2x_1)^{3-2i}=(2x_1)^3-3\epsilon (2x_1),\qquad {\text{so}}\qquad x_3=4x_1^3-3\epsilon x_1.
$$
\item[(iii)] $n\ge 4$. We have 
\begin{eqnarray*}
2x_n & = & \sum_{i=0}^{\lfloor n/2\rfloor} \frac{n}{n-i}\binom{n-i}{i}(-\epsilon)^i (2x_1)^{n-2i}\\
& = & (2x_1)^n-n\epsilon (2x_1)^{n-2}+\frac{n(n-3)}{2} (2x_1)^{n-4}+\sum_{i\ge 3}^{\lfloor n/2\rfloor} 
\frac{n(-\epsilon)^i}{n-i}\binom{n-i}{i}  (2x_1)^{n-2i}.
\end{eqnarray*}
\end{itemize}
\end{example}

The following variation of a result of Luca \cite{Luca} is useful. Let $P(m)$ denote the largest prime factor of the positive integer $m$.

\begin{lem}
\label{lem:Luc}
If $P(x_n)\le 5$, then either $n=1$, or $n=2$ and $x_2\in \{3,9,243\}$.
\end{lem} 

\begin{proof}
In \cite{Luca} it was shown that if $\varepsilon=1$ and $P(x_n)\le 5$, then $n=1$. We give here a proof for both cases $\epsilon\in \{\pm 1\}$. Since $x_n=y_{2n}/y_n$, where 
$y_m=(\delta^m-\eta^m)/(\delta-\eta)$, it follows, by Carmichael's Primitive Divisor Theorem \cite{Carm}, that if $n\ge 7$, then $x_n$ has  a prime factor which is primitive for $y_{2n}$. In particular, this prime is
$\ge 2n-1>5$. Thus, $n\le 6$. Assume next that $n>1$. If $n\in \{3,6\}$, then $x_n$ is of the form $x(4x^2\pm 3)$, where $x=x_{\ell}$ with $\ell=n/3\in \{1,2\}$. The factor $4x^2\pm 3$ is larger than $1$ 
(since $x_n>x_{\ell}$)  odd (hence, coprime to $2$), not a multiple of $9$, and coprime to $5$ since ${\displaystyle{\left(\frac{\pm 3}{5}\right)=-1}}$. Thus, the only possibility is $4x^2\pm 3=3$, equation which does not have a positive integer solution $x$. If $n\in \{2,4\}$, then $x_n=2x^2\pm 1$, where $x=x_{\ell}$ and $\ell=n/2\in \{1,2\}$. Further, if $\ell=2$ only the case with the $-1$ on the right is possible. The expression $2x^2-1$ is odd, and coprime to both $3$ and $5$ since ${\displaystyle{\left(\frac{2}{3}\right)=\left(\frac{2}{5}\right)=-1}}$, so the case $x_n=2x_{\ell}^2-1$ is not possible. Finally, if $x_n=2x_{\ell}^2+1$, then $n=2,~\ell=1$. Further, 
$2x^2+1$ is coprime to $2$ and $5$ so we must have $2x^2+1=3^b$ for some exponent $b$. Thus, $x^2=(3^b-1)/(3-1)$, and the only solutions are $b\in \{1,2,5\}$ by a result of Ljunggren \cite{Lju}. 
\end{proof}

Since none of $3,~9,~243$ are of the form $F_m^{(k)}$ for any $m\ge 1,~k\ge 4$, for our practical purpose we will use the implication that if $x_n=F_m^{(k)}$ and $P(x_n)\le 5$, then $n=1$.
e of $3,~9,~243$ are of the form $F_m^{(k)}$ for any $m\ge 1,~k\ge 4$, for our practical purpose we will use the implication that if $x_n=F_m^{(k)}$ and $P(x_n)\le 5$, then $n=1$.

%%%%%%%%%%%%%%%%%%%%%%%%%%%%%%%%%%%%%%%%%%
\section{A small linear form in logarithms}
%%%%%%%%%%%%%%%%%%%%%%%%%%%%%%%%%%%%%%%%%%
We assume that $ (x_1, y_1) $ is the fundamental solution of the Pell equation \eqref{Pelleqn}. As in Subsection \ref{subs:Pell}, we set
$$
x_1^2-dy_1^2=:\epsilon, \qquad \epsilon \in\{\pm 1\},
$$
and put
$$
\delta:=x_1+{\sqrt{d}}y_1\qquad {\text{\rm and  }}\qquad \eta:=x_1-{\sqrt{d}}y_1 =\epsilon \delta^{-1}.
$$
From \eqref{Pellsoln} (or \eqref{eq:Pellsol1}), we get
\begin{equation}\label{Pellxterm}
x_n=\dfrac{1}{2}\left(\delta^{n}+\eta^{n}\right).
\end{equation}
Since $ \delta\geq 1+\sqrt{2}>2>\alpha$, it follows that the estimate
\begin{eqnarray}\label{Pellenq}
\dfrac{\delta^{n}}{\alpha^2}\leq x_{n}<\delta^n\quad {\text{\rm holds for all}} \quad n\geq 1.
\end{eqnarray}
We now assume, as in the hypothesis of Theorem \ref{Main}, that $(n_1,m_1)$ and $(n_2,m_2)$ are pairs of positive integers with $n_1<n_2$, $2\le m_1<m_2$ and
$$
x_{n_1}=F_{m_{1}}^{(k)} \qquad{\text{\rm  and}} \qquad x_{n_2}=F_{m_2}^{(k)}.
$$
By setting $(n,m)=(n_j,m_j)$ for $ j\in\{1,2\}$ and using the inequalities \eqref{Fib12} and \eqref{Pellenq}, we get that
\begin{equation}
\label{bdding22}
\alpha^{m-2}\leq F_{m}^{(k)}=x_{n}< \delta^{n} \qquad {\text{\rm  and}}\qquad \dfrac{\delta^{n}}{\alpha^2}\leq x_{n}= F_{m}^{(k)}\leq \alpha^{m-1}.
\end{equation}
Hence, 
\begin{equation}
\label{eq:uuu}
nc_1\log\delta \leq m+1 \leq nc_1\log\delta +3, \qquad c_1 :=1/\log\alpha.
\end{equation}
Next, by using \eqref{Binet} and \eqref{Pellxterm}, we get
$$
\dfrac{1}{2}\left(\delta^{n}+\eta^{n}\right) = f_{k}(\alpha)\alpha^{m-1}+ (F_{m}^{(k)} -f_{k}(\alpha)\alpha^{m-1}),
$$
so
$$
\delta^{n}(2f_{k}(\alpha))^{-1}\alpha^{-(m-1)}-1=\dfrac{-\eta^n}{2f_{k}(\alpha)\alpha^{m-1}}+\dfrac{(F_{m}^{(k)} -f_{k}(\alpha)\alpha^{m-1})}{f_{k}(\alpha)\alpha^{m-1}}.
$$
Hence, by using \eqref{approxgap} and Lemma \ref{fala5}(i), we have
\begin{equation}
\label{bbding1}
|\delta^{n}(2f_{k}(\alpha))^{-1}\alpha^{-(m-1)}-1|\leq \dfrac{1}{\alpha^{m-1}\delta^{n}}+\dfrac{1}{\alpha^{m-1}}<\dfrac{1.5}{\alpha^{m-1}}.
\end{equation}
In the above, we have used the facts that $1/f_{k}(\alpha)< 2$, $|F_{m}^{(k)} -f_{k}(\alpha)\alpha^{m-1}|<1/2$, $|\eta|=\delta^{-1}$, as well as the fact that $\delta>2$. 
We let $\Lambda$ be the expression inside the absolute value of the left--hand side above. We put 
\begin{equation}
\label{eq:Gamma1}
\Gamma:=n\log\delta - \log(2f_{k}(\alpha))-(m-1)\log\alpha.
\end{equation}
Note that $e^{\Gamma}-1=\Lambda$. Inequality \eqref{bbding1} implies that
\begin{equation}
\label{eq:Gamma2}
|\Gamma|<\dfrac{3}{\alpha^{m-1}}.
\end{equation}
Indeed, for $m\ge 3$, we have that $\frac{1.5}{\alpha^{m-1}}<\frac{1}{2}$, and then inequality \eqref{eq:Gamma2} follows from \eqref{bbding1} via the fact that 
\begin{equation}
\label{eq:x2x}
|e^{\Gamma}-1|<x\qquad {\text{\rm implies}}\qquad |\Gamma|<2x,\qquad {\text{\rm whenever}}\qquad x\in (0,1/2),
\end{equation}
with $x:=\frac{1.5}{\alpha^{m-1}}$. When $m=2$, we have $x_{n}=F_{m}^{(k)}=1$, so $n=1$, $\epsilon=1$, $\delta=1+{\sqrt{2}}$, and then 
$$
|\Gamma|=|\log(1+{\sqrt{2}})-\log(2f_k(\alpha) \alpha)|<\max\{\log(1+{\sqrt{2}}),\log(2f_k(\alpha)\alpha)\}<\log 3<\frac{3}{\alpha},
$$
where we used the fact that $1<2f_k(\alpha) \alpha<3$ (see Lemma \ref{fala5}, (i)). Hence, inequality \eqref{eq:Gamma1} holds for all pairs $(n,m)$ with $x_n=F_m^{(k)}$ with $m\ge 2$.  

Let us recall what we have proved, since this will be important later-on.

\begin{lem}
\label{lem:Gamma}
If $(n,m)$ are positive integers with $m\ge 2$ such that $x_n=F_m^{(k)}$, then with $\delta=x_1+{\sqrt{x_1^2-\epsilon}}$, we have
\begin{equation}
\label{eq:Gamma3}
|n\log\delta - \log(2f_{k}(\alpha))-(m-1)\log\alpha|<\frac{3}{\alpha^{m-1}}.
\end{equation}
\end{lem}

\section{Bounding $n$ in terms of $m$ and $k$}

We next apply Theorem \ref{Matveev11} on the left-hand side of \eqref{bbding1}. First we need to check that 
$$
\Lambda=\delta^{n}(2f_{k}(\alpha))^{-1}\alpha^{-(m-1)}-1
$$
is nonzero. Well, if it were, then $ \delta^{n}=2f_{k}(\alpha)\alpha^{m-1}$. So, $2f_k(\alpha)=\delta^n\alpha^{-(m-1)}$ is a unit.  To see that this is not so, we perform a norm calculation of the element
$2f_k(\alpha)$ in ${\mathbb L}:={\mathbb Q}(\alpha)$. For $i\in \{2,\ldots,k\}$, we have that $|\alpha^{(i)}|<1$, so that, by the absolute value inequality, we have
\begin{eqnarray*}
|2f_k(\alpha^{(i)})|&=&\frac{2|\alpha^{(i)}-1|}{|2+(k+1)(\alpha^{(i)}-2)|}\\
&\le& \frac{4}{(k+1)(2-|\alpha^{(i)}|)-2}<\frac{4}{k-1}\le \frac{4}{5}\quad {\text{\rm for}}\quad k\ge 6.
\end{eqnarray*}
Thus, for $k\ge 6$, using also Lemma \ref{fala5} (i), we get
$$
|\mathcal{N}_{{\mathbb L}/{\mathbb Q}}(2f_k(\alpha))|<|2f_k(\alpha)|\prod_{i=2}^k |2f_k(\alpha^{(i)})|<\frac{3}{2} \left(\frac{4}{5}\right)^{k-1}\le \frac{3}{2}\left(\frac{4}{5}\right)^5<1.
$$
This is for $k\ge 6$. For $k=4,~5$ one checks that $|\mathcal{N}_{{\mathbb L}/{\mathbb Q}}(2f_k(\alpha))|<1$ as well. In fact, the norm of $2f_k(\alpha)$ has been computed (for all $k\ge 2)$ in \cite{FuLu}, and the formula is 
$$
|\mathcal{N}_{{\mathbb L}/{\mathbb Q}}(2f_k(\alpha))|=\frac{2^k(k-1)^2}{2^{k+1} k^k-(k+1)^{k+1}}.
$$
One can check directly that the above number is always smaller than $1$ for all $k\ge 2$ (in particular, for $k=4,5$). 
Thus, $\Lambda\neq 0$, and we can apply Theorem \ref{Matveev11}. We take
$$
t=3, \quad \gamma_{1}=\delta, \quad \gamma_2=2f_{k}(\alpha), \quad \gamma_3 = \alpha, \quad b_1=n, \quad b_2=-1, \quad b_3 = -(m-1).
$$
We take  $\mathbb{K} = \mathbb{Q}(\sqrt{d}, \alpha)$ which has degree $D\le 2k$. Since $\delta\geq 1+\sqrt{2} >\alpha$, the second inequality in \eqref{bdding22} tells us right-away that $n\le m$, so we can take $B:=m$. We have $h(\gamma_{1})=(1/2)\log \delta$ and $h(\gamma_{3})=(1/k)\log\alpha$. Further, 
\begin{equation}
\label{eq:hgamma2}
h(\gamma_2) = h(2f_k(\alpha))\leq h(2)+ h(f_k(\alpha))
< 3\log k +\log 2<4\log k
\end{equation}
by Lemma \ref{fala5} (ii). So, we can take $A_1:=k\log\delta $,  $A_2:= 8k\log k$ and $ A_3:= 2\log2 $. 
Now Theorem \ref{Matveev11} tells us that
\begin{eqnarray*}
\log|\Lambda|&>& -1.4\times 30^{6}\times 3^{4.5}\times (2k)^{2}(1+\log 2k)(1+\log m)
(k\log\delta)(8k\log k)(2\log 2),\\
&>&-1.6\times 10^{13}k^{4}(\log k)^2\log(\delta)(1+\log m).
\end{eqnarray*}
In the above, we used the fact that $k\ge 4$, therefore $2k\le k^{3/2}$, so 
$$
1+\log(2k)\le 1+1.5\log k<2.5\log k.
$$
By comparing the above inequality with inequality \eqref{bbding1}, we get
$$
(m-1)\log\alpha-\log 3< 1.6\times 10^{13}k^{4}(\log k)^2(\log\delta)(1+\log m).
$$
Thus,
$$
(m+1)\log\alpha <1.7\times 10^{13}k^{4}(\log k)^2(\log\delta)(1+\log m).
$$
Since $ \alpha^{m+1}\ge \delta^{n} $ by the second inequality in \eqref{bdding22}, we get that
\begin{equation}
\label{eq:n}
n < 1.7\times 10^{13}k^{4}(\log k)^2(1+\log m).
\end{equation}
Furthermore, since $\alpha >1.927$, we get
\begin{equation}
\label{eq:m}
m<2.6\times 10^{13}k^{4}(\log k)^2(\log\delta)(1+\log m).
\end{equation}
We now record what we have proved so far, which are estimates \eqref{eq:n} and \eqref{eq:m}.
\begin{lem}
\label{lemma11}
If $ x_{n}=F_{m}^{(k)} $ and $m\geq 2$, then
$$
n  <  1.7\times 10^{13}k^{4}(\log k)^2(1+\log m)~~{\text{and}}~~m  <2.6\times 10^{13}k^4(\log k)^2(\log\delta)(1+\log m).
$$
\end{lem}
Note that in the above bound, $n$ is bounded only in terms of $m$ and $k$ (but not $\delta$). 

\section{Bounding $m_1,~n_1,~m_2,~n_2$ in terms of $k$}

Next, we write down inequalities \eqref{eq:Gamma3} for both pairs $(n,m)=(n_j,m_j)$ with $j=1,2$, multiply the one for $j=1$ with $n_2$ and the one with $j=2$ with $n_1$, subtract them and apply the triangle 
inequality to the result to get that
\begin{eqnarray*}
&& |(n_2-n_1)\log(2f_{k}(\alpha))  - (n_1m_2-n_2m_1+n_2-n_1)\log\alpha| \\
& \leq & n_2|n_1\log\delta -\log(2f_k(\alpha))-(m_1-1)\log\alpha|\\
& + & n_1|n_2\log\delta -\log(2f_k(\alpha))-(m_2-1)\log\alpha|\\
& \leq & \dfrac{3n_2}{\alpha^{m_1-1}}+ \dfrac{3n_1}{\alpha^{m_2-1}} < \dfrac{6n_{2}}{\alpha^{m_1-1}}.
\end{eqnarray*} 
Therefore, we have
\begin{eqnarray}\label{ineqqq2}
\qquad\left|(n_2-n_1)\log(2f_k(\alpha)) -(n_1m_2-n_2m_1+n_2-n_1)\log\alpha\right|< \dfrac{6n_{2}}{\alpha^{m_1-1}}.
\end{eqnarray}
We are now set to apply Theorem \ref{Matveev12} with
\begin{eqnarray*}
\gamma_1=2f_{k}(\alpha), \quad \gamma_2=\alpha, \quad b_1= n_2-n_1,\quad b_2=-(n_1m_2-n_2m_1+n_2-n_1).
\end{eqnarray*}
The fact that $\gamma_1$ and $\gamma_2$ are multiplicatively independent follows because $\alpha$  is a unit and $2f_k(\alpha)$ isn't by a previous argument. 
Next, we observe that $ n_2-n_1 <n_2$, while by the absolute value of the inequality in \eqref{ineqqq2}, we have
$$
|n_1m_2-n_2m_1+n_2-n_1|\leq (n_2-n_1)\dfrac{\log(2f_k(\alpha))}{\log\alpha}+\dfrac{6n_2}{\alpha^{m_1-1}\log\alpha}<6n_2.
$$
In the above, we used that 
$$
\frac{\log(2f_k(\alpha))}{\log \alpha}<\frac{\log(1.5)}{\log \alpha}<1\qquad {\text{\rm and}}\qquad \frac{6}{\alpha^{m_1-1}\log\alpha}<5,
$$
because $\alpha\ge \alpha_4>1.92$ and $m_1\geq 2$. We take $\mathbb{K}:=\mathbb{Q}(\alpha)$ which has degree $D=k$. So, we can take
$$
\log B_1=4\log k>\max\left\{h(\gamma_{1}), \dfrac{|\log \gamma_{1}|}{k}, \dfrac{1}{k}\right\}
$$
(see inequality \eqref{eq:hgamma2}), 
and
$$
\log B_2 =\dfrac{1}{k}= \max\left\{h(\gamma_2), \dfrac{|\log \gamma_2|}{k}, \dfrac{1}{k}\right\}.
$$
Thus,
$$
b^{\prime}=\dfrac{(n_2-n_1)}{k\times(1/k)}+\dfrac{|n_1m_2-n_2m_1+n_2-n_1|}{4k\log k}<n_2+\dfrac{6n_2}{4k\log k}<1.3n_2.
$$
Now Theorem \ref{Matveev12} tells us that with
$$
\Gamma:=(n_2-n_1)\log(2f_k(\alpha)) -(n_1m_2-n_2m_1+n_2-n_{1})\log\alpha,
$$
we have
$$
\log|\Gamma|>-24.34\times k^4\left(\max\left\{\log(1.3n_2)+0.14, \dfrac{21}{k}, \dfrac{1}{2}\right\}\right)^{2}(4\log k)\left(\dfrac{1}{k}\right).
$$
Thus,
$$
\log|\Gamma|>-97.4 k^3\log k\left(\max\left\{\log(1.5n_2), \dfrac{21}{k}, \dfrac{1}{2}\right\}\right)^{2},
$$
where we used the fact that $\log(1.3n_2)+0.14=\log(1.3\times e^{0.14} n_2)<\log(1.5 n_2)$. 
By combining the above inequality with \eqref{ineqqq2}, we get
\begin{equation}
\label{eq:choices}
(m_1-1)\log\alpha - \log(6n_2)<97.4k^3\log k\left(\max\left\{\log(1.5n_2), \dfrac{21}{k}, \dfrac{1}{2}\right\}\right)^{2}.
\end{equation}
Since $\log(1.5n_2)\ge \log 3>1.098$, the maximum in the right--hand side above cannot be  $1/2$.  If it is not  $\log(1.5 n_2)$, we then get
\begin{equation}
\label{eq:n1n2}
1.098<\log(1.5n_2)\leq \frac{21}{k}\leq 5.25,\qquad {\text{\rm so}}\quad  k\le 19\quad {\text{\rm and}}\quad n_2\leq 127.
\end{equation}
Then, the above inequality \eqref{eq:choices} gives
\begin{eqnarray}
\label{1withalpha}
(m_1+1)\log\alpha & < & 97.4\times 21^2 k\log k +\log(6\times 127)+2\log \alpha\nonumber\\
& < & 4.3\times 10^5 k\log k.
\end{eqnarray}
Since $\alpha\ge 1.927$, we get that
\begin{equation}
\label{eq:branch1}
m_1+1<6.6\times 10^5 k\log k.
\end{equation}
Further, we have
\begin{eqnarray*}
(\alpha^{(m_1+1)})^{n_2} & > & (3F_{m_1}^{(k)})^{n_2}\ge (2F_{m_1}^{(k)}+1)^{n_2}=(2x_{n_1}+1)^{n_2}\\
& = & (\delta^{n_1}+(1+\eta^{n_1}))^{n_2}>\delta^{n_1 n_2}=(\delta^{n_2})^{n_1}\\
& = & (2x_{n_2}-\eta^{n_2})^{n_1}>2x_{n_2}-1>x_{n_2}=F_{m_2}^{(k)}>  \alpha^{m_2-2},
\end{eqnarray*}
so
\begin{equation}
\label{eq:m22}
m_2\le 1+n_2(m_1+1)<8.4\times 10^7k\log k.
\end{equation}
Since $n_1<n_2$, inequalities \eqref{eq:n1n2}, \eqref{eq:branch1} and \eqref{eq:m22} bound $m_1,n_1,m_2,n_2$ in terms of $k$ when the maximum in the right--hand side of \eqref{eq:choices} 
is $21/k$. 

Assume next that the maximum in the right--hand side of \eqref{eq:choices} is $\log(1.5n_2)$. Then
\begin{eqnarray}
\label{eq:branch2}
(m_1+1)\log\alpha & < & 97.4k^{3}\log k (\log(1.5n_{2}))^{2}+2\log \alpha+\log(6n_2)\nonumber\\
& < & 97.4 k^3 (\log k) (\log 1.5+\log n_2)^2+\log(24n_2)\nonumber\\
& < & 97.5\times 2.56 k^3 (\log k) (\log n_2)^2+6\log n_2\nonumber\\
& < & 249.6 k^3 (\log k) (\log n_2)^2+6\log n_2\nonumber\\
& < & 249.6 k^3(\log k)(\log n_2)^2\left(1+\frac{6}{249.6 k^3(\log k)(\log n_2)}\right)\nonumber\\
& < & 2.5\times 10^2 k^3 (\log k) (\log n_2)^2.
\end{eqnarray}
For the above inequality, we used that $2\log \alpha+\log(6n_2)<\log(24n_2)\le 6\log n_2$ (since $n_2\ge 2$ and $\alpha<2$), the fact that $\log(1.5 n_2)<1.6\log n_2$ holds for $n_2\ge 2$ and the fact that
$$
1+\frac{6}{249.6 k^3(\log k)(\log n_2)}<1.0004\qquad {\text{\rm holds~for}}\quad k\ge 4\quad {\text{\rm and}}\quad n_2\ge 2.
$$
In turn, since $\alpha\ge \alpha_4\ge 1.927$, \eqref{eq:branch2} yields
\begin{equation}
\label{ineqqu231}
m_1<4\times 10^2k^3(\log k) (\log n_2)^2.
\end{equation}
Since $ \alpha^{m_1+1}>\delta^{n_{1}} \geq \delta$ (see the second relation in \eqref{bbding1}), we get
\begin{equation}
\label{eq:inter}
\log\delta\le n_1\log\delta <(m_1+1)\log\alpha < 2.5\times 10^2 k^3(\log k)(\log n_{2})^{2}.
\end{equation}
By combining the above inequality with Lemma \ref{lemma11} for $(n,m):=(n_2,m_2)$ together with the fact that $n_2<m_2$, we get
\begin{eqnarray}
\label{eq:finalm21}
m_2 & < & 2.6\times 10^{13} k^4 (\log k)^2 (\log \delta) (1+\log m_2)\nonumber\\ 
& < & 2.6\times 10^{13} k^4 (\log k)^2 (2.5\times 10^2 k^3 (\log k))(\log m_2)^2 (1.92\log m_2)\nonumber\\
& < & 1.25\times 10^{16} k^7 (\log k)^3 (\log m_2)^3.
\end{eqnarray}
In the above, we used that $1+\log m_2\le 1.92 \log m_2$ holds for all $m_2\ge 3$. 
We now apply Lemma \ref{gl} with $m:=3$ and $T:=1.25\times 10^{16} k^7 (\log k)^3$ (which satisfies  the hypothesis $T>(4\cdot m^2)^m$), to get
\begin{eqnarray}
\label{eq:finalm2}
m_2 & < & 8\times 1.25\times 10^{16} k^7 (\log k)^3 (\log T)^3\nonumber\\
& < & 10^{17} k^7 (\log k)^3(7\log k+3\log\log k+\log(1.25\times 10^{16}))^3\nonumber\\
& < & 10^{17}\times (4.1\times 10^5) k^7 (\log k)^6\nonumber\\
& < &  4.1\times 10^{22} k^7 (\log k)^6.
\end{eqnarray}
In the above calculation, we used that 
$$
\left(\frac{7\log k+3\log\log k+\log(10^{16})}{\log k}\right)^3<4.1\times 10^5\qquad {\text{\rm for~all}}\qquad k\ge 4.
$$  
By substituting  the upper bound \eqref{eq:finalm2} for $m_2$ in the first inequality of Lemma \ref{lemma11}, we get
\begin{eqnarray}
\label{eq:finaln2}
n_ 2 & < & 1.7\times 10^{13} k^4 (\log k)^2 (1+\log m_2)\nonumber\\
& < & 1.7\times 10^{13} k^4 (\log k)^2(1+\log(4.1\times 10^{22})+7\log k+6\log\log k)\nonumber\\
& < & 1.7\times 10^{13} \times 48 k^4 (\log k)^3\nonumber\\
& < & 8.2\times 10^{14} k^4 (\log k)^3,
\end{eqnarray}
where we used the fact that
$$
\frac{7\log k+6\log\log k+\log(4.1\times 10^{22})+1}{\log k}<48\quad {\text{\rm for~all}}\quad k\ge 4.
$$
Finally, if we substitute the upper bound \eqref{eq:finaln2} for $n_2$ into the inequality \eqref{eq:branch2}, we get
\begin{eqnarray}
\label{eq:finalm1alpha}
(m_1+1)\log \alpha & < & 2.5\times 10^2 k^3 (\log k) (\log n_2)^2\nonumber\\
& < & 2.5\times 10^2 k^3 (\log k) (1+\log(4\times 10^{16})+4\log k+3\log\log k)^2\nonumber\\
& < & 2.5 \times 10^2 (9.2\times 10^2) k^3 (\log k)^3\nonumber\\
& < & 2.3\times 10^5 k^3 (\log k)^3.
\end{eqnarray}
In the above, we used that 
$$
\left(\frac{4\log k+3\log\log k+\log(3.4\times 10^{16})+1}{\log k}\right)^2<9.2\times 10^2\qquad {\text{\rm for~all}}\qquad k\ge 4.
$$
Thus, using $\alpha>1.927$, we get 
\begin{equation}
\label{eq:finalm1}
m_1<3.6\times 10^5 k^3(\log k)^3.
\end{equation}
Thus, inequalities \eqref{eq:finalm2}, \eqref{eq:finaln2}, \eqref{eq:finalm1} give upper bounds for $m_2,~n_2$ and $m_1$, respectively, in the case in which 
the maximum in the right--hand side of inequality \eqref{eq:choices} is $\log(1.5 n_2)$.  Comparing inequalities \eqref{eq:finalm2} with \eqref{eq:m22}, \eqref{eq:finaln2} with
\eqref{eq:n1n2}, and \eqref{eq:finalm1alpha} with \eqref{eq:branch1}, respectively, we conclude that \eqref{eq:finalm2}, \eqref{eq:finaln2} and \eqref{eq:finalm1} always hold.
 Let us summarise what we have proved again, which are the bounds \eqref{eq:finalm2}, \eqref{eq:finaln2} and \eqref{eq:finalm1}.
\begin{lem}
\label{lemma12}
If $x_{n_{j}}= F_{m_{j}}^{(k)}$ for $ j\in\{1,2\} $ with $2\le m_1<m_2$, and $n_1<n_2$, then
\begin{eqnarray*}
m_1<3.6\times 10^5 k^3(\log k)^3,\quad m_2<4.1\times 10^{22}k^7(\log k)^6,\quad n_2 < 8.2\times 10^{14}k^4(\log k)^3.
\end{eqnarray*}
\end{lem}
Since $n_1\le m_1$, the above lemma gives bounds for all of $m_1,n_1,m_2,n_2$ in terms of $k$ only. 

\section{The case $k>500$}

\begin{lem}
If $k>500$, then 
\begin{equation}
\label{eq:1}
8m_2^3 <2^k.
\end{equation}
\end{lem}

\begin{proof}
In light of the upper bound given by Lemma \ref{lemma12} on $m_2$, this is implied by
$$
4.1\times 10^{22} k^7(\log k)^6<2^{k/3-1},
$$
which indeed holds for all $k\ge 462$ as confirmed by {\it Mathematica}. 
\end{proof}

From now on, we assume that  $k>500$. Thus, \eqref{eq:1} holds. The main result of this section is the following.

\begin{lem}
\label{lem:maink>500}
If $k>500$, then $m_1\le k+1$. In particular, $x_{n_1}=F_{m_1}^{(k)}=2^{m_1-2}$, and $n_1=1$. 
\end{lem}

For the proof, we go to Lemma \ref{Gomez} and write for $m:=m_j$ with $j=1,2$ the following approximations
\begin{equation}
\label{eq:approxFm}
F_m^{(k)}=2^{m-2}(1+\zeta_m)=2^{m-2}\left(1+\delta_m\left(\frac{k-m}{2^{k+1}}\right)+\gamma_m\right),
\end{equation}
where $\delta_m\in \{0,1\}$ and
\begin{eqnarray}
\label{eq:zeta}
|\zeta_m| & \le & \frac{m}{2^{k+1}}+\frac{m^2}{2^{2k+2}}+\frac{4m^3}{2^{3k+3}}<\frac{1}{2^{2k/3}}\left(\frac{1}{2}+\frac{1}{2^{2+2k/3}}+\frac{1}{2^{4+4k/3}}\right)<\frac{1}{2^{2k/3}},\\
|\gamma_m| & \le & \frac{m^2}{2^{2k+2}}+\frac{4m^3}{2^{3k+3}}<\frac{1}{2^{4k/3}}\left(\frac{1}{2^2}+\frac{1}{2^{2k/3+4}}\right)<\frac{1}{2^{4k/3}}\nonumber,
\end{eqnarray}
where we used that $m<2^{k/3-1}$ (see \eqref{eq:1}) and $k\ge 4$. We then write
$$
|F_m^{(k)}-x_n|=0,
$$
from where we deduce 
\begin{equation}
\label{eq:zeta1}
|2^{m-1}(1+\zeta_m)-\delta^n|=\frac{1}{\delta^n}.
\end{equation}
Thus,
$$
|2^{m-1}-\delta^n|=\frac{1}{\delta^n}+|\zeta_m| 2^{m-1},
$$
so
\begin{equation}
\label{eq:20}
|1-\delta^n 2^{-(m-1)}|=\frac{1}{2^{m-1}\delta^n}+|\zeta_m|<\frac{1}{2^m}+\frac{1}{2^{2k/3}}\le \frac{1}{2^{\min\{2k/3-1,m-1\}}}.
\end{equation}
In the above, we used that $\delta^n\ge \delta\ge 1+{\sqrt{2}}>2$. The right--hand side above is $<1/2$, so we may pass to logarithmic form as in \eqref{eq:x2x} to get that
\begin{equation}
\label{eq:3}
\left|n\log \delta-(m-1)\log 2\right|<\frac{1}{2^{\min\{2k/3-2,m-2\}}}.
\end{equation}
We write the above inequality for $(n_1,m_1)$ and $(n_2,m_2)$ cross-multiply the one for $(n_1,m_1)$ by $n_2$ and the one for $(n_2,m_2)$ by $n_1$ and subtract them to get
$$
|(n_1(m_2-1)-n_2(m_1-1))\log 2|<\frac{n_2}{2^{\min\{2k/3-2,m_1-2\}}}+\frac{n_1}{2^{\min\{2k/3-2,m_2-2\}}}.
$$
Assume $n_1(m_2-1)\ne n_2(m_1-1)$. Then the left--hand side above is $\ge \log 2>1/2$. In particular, either
$$
2^{\min\{2k/3-2,m_1-2\}}<4n_2\quad {\text{\rm or}}\qquad 2^{\min\{2k/3-2,m_2-2\}}<4n_1.
$$
The first one is weaker than the second one and is implied by the second one, so the first one must hold. If the minimum is $2k/3-2$, we then get
$$
2^{2k/3-2}\le 4n_2<2^{k/3+1},
$$
because $n_2\le m_2<2^{k/3-1}$, so $2k/3-2<k/3+1$, or $k<9$, a contradiction. Thus, 
$$
2^{m_1-2}<4n_2<2^{k/3+1},
$$
getting
$$
m_1<k/3+3<k+2.
$$
Thus, by Example \ref{exa:1} (i), we get that $x_{n_1}=F_{m_1}^{(k)}=2^{m_1-2}$, which by Lemma \ref{lem:Luc}, implies that $n_1=1$. 

So, we got the following partial result.

\begin{lem}
\label{lem13}
For $k>500$, either $n_1=1$ and $m_1<k/3+3$, or $n_1/n_2=(m_1-1)/(m_2-1)$. 
\end{lem}

To finish the proof of Lemma \ref{lem:maink>500}, assume for a contradiction that $m_1\ge k+2$.  
Lemma \ref{lem13} shows that $n_1/n_2=(m_1-1)/(m_2-1)$. Further, in \eqref{eq:approxFm}, we have $\delta_{m_1}=\delta_{m_2}=1$. 
Thus, we can rewrite equation \eqref{eq:zeta1} using $\gamma_m$ for both $m\in \{m_1,m_2\}$. We get
$$
\left|2^{m-1}\left(1+\frac{k-m}{2^{k+1}}+\gamma_m\right)-\delta^n\right|=\frac{1}{\delta^n},
$$ 
so
$$
\left|2^{m-1}\left(1+\frac{k-m}{2^{k+1}}\right)-\delta^n\right|\le \frac{1}{\delta^n}+2^{m-1}|\gamma_m|,
$$
therefore
$$
\left|\left(1+\frac{k-m}{2^{k+1}}\right)-\delta^n 2^{-(m-1)}\right|\le \frac{1}{2^{m-1}\delta^n}+|\gamma_m|.
$$
Now $\delta^n\ge \alpha^{m-2}$  by the  first inequality in \eqref{bdding22}. Thus,
$$
2^{m-1}\delta^n\ge 2^{m-1} \alpha^{m-2}\ge 2^{m-1} 2^{0.9(m-2)}>2^{1.9m-3}>2^{1.9k}>2^{4k/3},
$$
where we used the fact that $m\ge k+2$ and that $\alpha\ge \alpha_4=1.9275\ldots>2^{0.9}$. Since also $|\gamma_m|\le \frac{1}{2^{4k/3}}$, we get that 
$$
\left|\left(1+\frac{k-m}{2^{k+1}}\right)-\delta^n 2^{-(m-1)}\right|<\frac{2}{2^{4k/3}}.
$$
The expression $1+(k-m)/2^{k+1}$ is in $[1/2,2]$. Thus, 
$$
\left|1-\delta^n 2^{-(m-1)} (1+(k-m)/2^{k+1})^{-1}\right|<\frac{4}{2^{4k/3}}.
$$
The right--hand side is $<1/2$ for all $k\ge 4$. We pass to logarithms via implication \eqref{eq:x2x} getting that
$$
\left|n\log \delta-(m-1)\log 2-\log\left(1+\frac{k-m}{2^{k+1}}\right)\right|<\frac{8}{2^{4k/3}}.
$$
We evaluate the above in $(n,m):=(n_j,m_j)$ for $j=1,2$. We multiply the expression for $j=1$ with $n_2$, the one with $j=2$ with $n_1$, subtract them 
and use $n_2(m_1-1)=n_1(m_2-1)$, to get
\begin{equation}
\label{eq:2}
\left|n_1\log\left(1+\frac{k-m_2}{2^{k+1}}\right)-n_2\left(1+\frac{k-m_1}{2^{k+1}}\right)\right|<\frac{16n_2}{2^{4k/3}}.
\end{equation}
One checks that in our range we have
\begin{equation}
\label{eq:16}
16 n_2<2^{k/4}.
\end{equation}
By Lemma \ref{lemma12}, this is fulfilled if 
$$
16\times 8.2\times 10^{14} k^4 (\log k)^3<2^{k/4},
$$
and \textit{Mathematica} checks that this is so for all $k\ge 346$. Thus,  inequality \eqref{eq:2} implies 
$$
\left|n_1\log\left(1+\frac{k-m_2}{2^{k+1}}\right)-n_2\left(1+\frac{k-m_1}{2^{k+1}}\right)\right|<\frac{2^{k/4}}{2^{4k/3}}<\frac{1}{2^{13k/12}}.
$$
Using the fact that the inequality
$$
|\log(1+x)-x|<2x^2\quad {\text{\rm holds~for}}\quad |x|<1/2,
$$
with $x_j:=(k-m_j)/2^{k+1}$ for $j=1,2$, and noting that $2x_j^2<2m_2^2/2^{2k+2}$ holds for both $j=1,2$, we get 
$$
\left|\frac{n_1(k-m_2)}{2^{k+1}}-\frac{n_2(k-m_1)}{2^{k+1}}\right|<\frac{4 n_2m_2^2}{2^{2k+2}}+\frac{1}{2^{13k/12}}.
$$
In the right--hand side, we have
$$
\frac{4n_2m_2^2}{2^{2k+2}}<\frac{2^{2+(k/4-4)+2(k/3-1)}}{2^{2k+2}}=\frac{1}{2^{13k/12+5}}.
$$
Hence,
$$
\left|\frac{n_1(k-m_2)}{2^{k+1}}-\frac{n_2(k-m_1)}{2^{k+1}}\right|<\frac{2}{2^{13k/12}}.
$$
which implies 
$$
|n_1(k-m_2)-n_2(k-m_1)|<\frac{4}{2^{k/12}}.
$$
Since $k>500$, the right--hand side is smaller than $1$. Since the left--hand side is an integer, it must be the zero integer. Thus, 
$$
n_1/n_2=(k-m_1)/(k-m_2).
$$
Since also $n_1/n_2=(m_1-1)/(m_2-1)$, we get that $(m_1-1)/(m_2-1)=(m_1-k)/(m_2-k)$, or $(m_1-1)/(m_1-k)=(m_2-1)/(m_2-k)$. This gives $1+(k-1)/(m_1-k)=1+(k-1)/(m_2-k)$, 
so $m_1=m_2$, a contradiction.

Thus, $m_1\le k+1$. By Example \ref{exa:1} (i), we get that $x_{n_1}=2^{m_1-2}$, which by Lemma \ref{lem:Luc} implies that $n_1=1$. This finished the proof of Lemma \ref{lem:maink>500}.

\section{The case $m_1>376$}

Since $k>500$, we know, by Lemma \ref{lem:maink>500}, that $m_1\le k+1$ and $n_1=1$. In this section, we prove that if also $m_1>376$, then the only solutions 
are the ones shown at (i) and (ii) of the Theorem \ref{Main}. This finishes the proof of Theorem \ref{Main} in the case $k>500$ and $m_1>376$. The remaining cases are handled computationally 
in the next section. 

\subsection{A lower bound for $m_1$ in terms of $m_2$}

The main goal of this subsection  is to prove the following result. 

\begin{lem}
\label{lem:7}
Assume that $m_1>376$. Then $2^{m_1-6}>\max\{k^4,n_2^2\}$.
\end{lem}

\begin{proof}
Assume $m_1>376$. We evaluate \eqref{eq:20} in $(n,m):=(n_2,m_2)$. Further, by Lemma \ref{lem:Luc}, $x_{n_2}$ is not a power of $2$,  
so $m_2\ge k+2$, therefore $\min\{2k/3-2,m_2-2\}=2k/3-2$, getting
\begin{equation}
\label{eq:41}
|n_2\log \delta-(m_2-1)\log 2|<\frac{1}{2^{2k/3-2}}.
\end{equation}
We write a lower bound for the left--hand side using Theorem \ref{Matveev12}. Let
\begin{equation}
\label{eq:Lambda10}
\Lambda:=n_2\log \delta-(m_2-1)\log 2.
\end{equation}
We have
$$
\gamma_1=\delta,\quad \gamma_2=2,\quad b_1=n_2,\quad b_2=-(m_2-1).
$$
We have ${\mathbb K}:={\mathbb Q}(\delta)$ has $D=2$. Further, $h(\gamma_1)=(\log \delta)/2$ and $h(\gamma_2)=\log 2$. Thus, we can take $\log B_1=(\log \delta)/2$, $\log B_2=\log 2$, 
$$
b'=\frac{n_2}{2\log 2}+\frac{m_2-1}{\log \delta}<m_2\left(\frac{1}{2\log 2}+\frac{1}{\log(1+{\sqrt{2}})}\right)<2m_2.
$$
Further, Theorem \ref{Matveev12} is applicable since $\gamma_1,~\gamma_2$ are real positive and multiplicatively independent (this last condition follows because $\delta$ is a unit and $2$ isn't). 
Theorem \ref{Matveev12} shows that
$$
\log| \Lambda |>-24.34\cdot 2^4 E^2 (\log \delta/2) \log 2>-195 \log 2 (\log \delta) E^2,~~ E:=\max\{\log(3m_2), 10.5\}^2,
$$
where we used $\log(3m_2)>0.14+\log(2m_2)>0.14+\log b'$. Thus, 
\begin{equation}
\label{eq:51}
|\Lambda|>2^{-195 (\log \delta) E^2}.
\end{equation}
Comparing \eqref{eq:41} and \eqref{eq:51}, we get
\begin{equation}
\label{eq:boundLambda}
195 (\log \delta) E^2>2k/3-2.
\end{equation}
Since
$$
2^{m_1-1}=2x_1=\delta+\frac{\varepsilon}{\delta}>\frac{\delta}{2},
$$
we get $\delta<2^{m_1}$, so $\log \delta<m_1\log 2$. Thus, 
$$
(m_1\log 2) (195E^2)>2k/3-2.
$$
Now let us assume that in fact the inequality $2^{m_1-6}<\max\{k^4,n_2^2\}$ holds. Assume first that the above maximum is $n_2^2$. Then $m_1\log 2<\log (2^6n_2^2)$. We thus get that
$$
2k/3-2<195\log(64n_2^2) E^2.
$$
Since by Lemma \ref{lemma12}, $64n_2^2 < 64\times 8.2^2\times 10^{28} k^8(\log k)^6$, and $3m_2 < 12.3\times 10^{22} k^7 (\log k)^5$, we get that
$$
2k/3-2<195 \log (64\times 8.2^2\times 10^{28} k^8(\log k)^6) \max\{10.5, \log(12.3\times 10^{22} k^7 (\log k)^5)\}^2,
$$
which gives $k<4\times 10^{9}$. Thus, 
$$
n_2<8.2 \times 10^{14} k^4 (\log k)^3<5\times 10^{55},
$$
and since 
$$
2^{m_1-6}\le n_2^2<(5\times 10^{55})^2,
$$
we get $m_1<6+2(\log 5\times 10^{55})/(\log 2)<377$, contradicting the fact that $m_1>376$. This was in the case $n_2\ge k^2$. But if $n_2<k^2$, then $\max\{n_2^2,k^4\}=k^4$ and the same argument gives us an even smaller bound on $k$; hence, on $m_1$. This contradiction finishes the proof of this lemma.
\end{proof}

\subsection{We have $m_2-1=n_2(m_1-1)$}

The aim of this subsection is to prove the following result.

\begin{lem}
\label{n1=1}
If $k>500$ and $m_1>376$, then $n_2(m_1-1)=m_2-1$. 
\end{lem}

For the proof, we write
\begin{eqnarray*}
2x_{1}=\delta+\frac{\epsilon}{\delta}=2F_{m_1}^{(k)}=2^{m_1-1};\\
2x_{n_2}=\delta^{n_2}+\left(\frac{\epsilon}{\delta}\right)^{n_2}=2F_{m_2}^{(k)}.
\end{eqnarray*}
Thus, 
\begin{eqnarray*}
2F_{m_2}^{(k)} & = & \sum_{i=0}^{\lfloor n_2/2\rfloor} \frac{n_2}{n_2-i} \binom{n_2-i}{i} (-\epsilon)^i 2^{(m_1-1)(n_2-2i)}\\
& = & 2^{(m_1-1)n_2} \left(1+\sum_{i=1}^{\lfloor n_2/2\rfloor} \frac{n_2}{n_2-i}\binom{n_2-i}{i}\left(-\frac{\epsilon}{2^{2(m_1-1)}}\right)^{i}\right).
\end{eqnarray*}
Note that
$$
\frac{n_2}{n_2-i}\binom{n_2-i}{i}<n_2^i.
$$
Thus, 
\begin{equation}
\label{eq:6}
\left|\frac{n_2}{n_2-i}\binom{n_2-i}{i}\left(-\frac{\epsilon}{2^{2(m_1-1)}}\right)^{i}\right|<\left(\frac{n_2}{2^{2(m_1-1)}}\right)^i.
\end{equation}
Since $m_1>376$, we have $2^{m_1-6}>n_2^2$ by Lemma \ref{lem:7}. In this case, \eqref{eq:6} tells us that
\begin{equation}
\label{eq:7}
\left|\frac{n_2}{n_2-i}\binom{n_2-i}{i}\left(-\frac{\epsilon}{2^{2(m_1-1)}}\right)^{i}\right|<\frac{1}{2^{1.5 m_1 i}} \left(\frac{n_2}{2^{0.5m_1-2}}\right)^i<\frac{1}{2^{1.5 m_1 i}}\left(\frac{1}{2^i}\right).
\end{equation}

Combining \eqref{eq:7} with \eqref{eq:zeta}, 
\begin{eqnarray*}
2x_{n_2} & = & 2^{(m_1-1)n_2} \left(1+\sum_{i=1}^{\lfloor n_2/2\rfloor} \frac{n_2}{n_2-i}\binom{n_2-i}{i}\left(-\frac{\epsilon}{2^{2(m_1-1)}}\right)^{i}\right)\\
& := & 2^{(m_1-1)n_2}(1+\zeta_{n_2}')\\ 
2F_{m_2}^{(k)} & = & 2^{m_2-1}\left(1+\zeta_{m_2}\right),
\end{eqnarray*}
where
$$
\zeta_{n_2}':=\sum_{i=1}^{\lfloor n_2/2\rfloor} \frac{n_2}{n_2-i}\binom{n_2-i}{i}\left(-\frac{\epsilon}{2^{2(m_1-1)}}\right)^{i}.
$$
Since $2x_{n_2}=2F_{m_2}^{(k)}$, we then have
$$
|2^{(m_1-1)n_2}-2^{m_2-1}|\le 2^{(m_1-1)n_2}|\zeta_{n_2}'|+2^{m_2-1} |\zeta_{m_2}|.
$$
If $(m_1-1)n_2\ne m_2-1$, then putting $R:=\max\{2^{(m_1-1)n_2},2^{m_2-1}\}$, the left--hand side above is $\ge R/2$, while the right-side above is $<R/2$, since 
$$
|\zeta_{m_2}|<\frac{1}{2^{2k/3}}<\frac{1}{4}\quad {\text{\rm and}}\quad |\zeta_{n_2}'|<\sum_{i\ge 1} \frac{1}{2^{1.5 m_1 i}}\left( \frac{1}{2^i}\right)<\frac{1}{2^{1.5 m_1}}\sum_{i\ge 1} \frac{1}{2^i}<\frac{1}{2^{1.5 m_1}}<\frac{1}{4}.
$$
This contradiction shows that $m_2-1=n_2(m_1-1)$, which finishes the proof of Lemma \ref{n1=1}.

\subsection{The case $n_2=2$}

By Lemma  \ref{n1=1}, we get $m_2=2m_1-1$. Since $m_1\le k+1$, we get that $m_2\le 2k+1$. Also, $m_2\ge k+2$. By Example \ref{exa:1} (ii), we have
$$
F_{m_2}^{(k)}=2^{m_2-2}-(m_2-k)2^{m_2-k-3}=x_2=2x_1^2-\epsilon=2(2^{m_1-2})^2-\epsilon.
$$
We thus get
$$
2^{2m_1-3}-(2m_1-k-1)2^{2m_1-k-4}=2^{2m_1-3}-\epsilon.
$$
We get that the $\epsilon=1$, and further $(2m_1-k-1)2^{2m_1-k-4}=1$, so $m_1=(k+3)/2$. This gives the parametric family (i) from Theorem \ref{Main}. 

\subsection{The case $n_2=3$}

By Lemma \ref{n1=1}, we get $m_2=3(m_1-1)+1=3m_1-2$. Since $m_1\le k+1$, we get that $m_2=3m_1-2\le 3k+1$. Further, $m_2\ge k+2$. If $m_2\in [k+2,2k+2]$, then, by Example \ref{exa:1} (ii), 
we have
$$
F_{m_2}^{(k)}=2^{m_2-2}-(m_2-k)2^{m_2-k-3}=x_3=4x_1^3-3\epsilon x_1=4(2^{m_1-2})^3-3\epsilon 2^{m_1-2},
$$
so $\epsilon=1$, and $(3m_1-k-2)2^{3m_1-k-5}=3 \times 2^{m_1-2}$. This gives
$$
(3m_1-k-2) 2^{2m_1-k-3}=3.
$$
By unique factorisation, we get
$$
3m_1-k-2=3\times 2^{a}\qquad {\text{\rm and}}\qquad 2m_1-k-3=-a
$$
for some integer $a\ge 0$. Solving, we get 
\begin{eqnarray*}
m_1 & = & 3\times 2^{a}+a-1,\\
k & = & 3\times 2^{a+1}+3a-5,
\end{eqnarray*}
and then $m_2=3m_1-2=9\times 2^{a}+3a-5$. The case $a=0$ gives $k=1$, which is not convenient so $a\ge 1$. 
This is the parametric family (ii). 

It can also be the case that $m_2\in [2k+3,3k+1]$. By Example \ref{exa:1} (iii), we get
$$
4(2^{m_1-2})^3-3\epsilon 2^{m_1-2}=2^{m_2-2}-(m_2-k)2^{m_2-k-3}-(m_2-2k+1)(m_2-2k-2) 2^{m_2-2k-5}.
$$
This leads to
$$
3\epsilon 2^{m_1-2}=(3m_1-k-2) 2^{3m_1-k-5}-(3m_1-2k-1)(3m_1-2k-4)2^{3m_1-2k-7}.
$$
Simplifying $2^{3m_1-2k-7}$ from both sides of the above equation we get
$$
3\epsilon 2^{2k+5-2m_1}=(3m_1-k-2)2^{k+2}-(3m_1-2k-1)(3m_1-2k-4).
$$
Since $m_2=3m_1-2\ge 2k+3$, it follows that $m_2\ge (2k+5)/3$, so $2k+5-2m_1\le (2k+5)/3$. It thus follows, by the absolute value inequality, that
\begin{eqnarray*}
2^{k+2} & < & (3m_1-k-2)2^{k+2}\le 3\cdot 2^{2k+5-2m_1}+(3m_1-2k-1)(3m_1-2k-4)\\
& \le & 3\cdot 2^{(2k+5)/3}+(k+2)(k-1),
\end{eqnarray*}
an inequality which fails for $k\ge 5$. Thus, there are no other solutions in this range for $n_2=3$ except for the ones indicated in (ii) 
of Theorem \ref{Main}.

\subsection{The case $n_2=4$}
In this case, we have $m_2=4(m_1-1)+1=4m_1-3$. Since $m_1\le k+1$, we have $m_2\le 4k+1$. Note that
\begin{equation}
\label{eq:x4}
x_4=2x_2^2-1=2(2x_1^2-\epsilon)^2-1 =  8x_1^4-8\epsilon x_1^2+1 =  8(2^{m_1-2})^4-8\epsilon (2^{m_1-2})^2+1
\end{equation}
is odd. Assume first that $m_2\in [k+2,2k+2]$. We then have, by Example \ref{exa:1},
\begin{equation}
\label{eq:m21}
F_{m_2}^{(k)} = 2^{m_2-2}-(m_2-k)2^{m_2-k-3}=2^{4m_1-5}-(4m_1-k-3)2^{4m_1-k-6}.
\end{equation}
Comparing \eqref{eq:m21} with \eqref{eq:x4}, we get
$$
(4m_1-k-3)2^{4m_1-k-6}=\epsilon 2^{2m_1-1}-1.
$$
First, $\epsilon=1$. Second, the right--hand side above is odd. This implies that the left--hand side is also odd. Thus, the left--hand side is in $\{1,3\}$. This is impossible since the right--hand side is at least $2^{753}$.  Thus, this instance does not give us any solution. 

Assume next that $m_2\in [2k+3,3k+3]$. Then
\begin{eqnarray*}
F_{m_2}^{(k)} & = & 2^{m_2-2}-(m_2-k)2^{m_2-k-3}+(m_2-2k+1)(m_2-2k-2)2^{m_2-2k-5}\\
& = & 8(2^{m_1-2})^4-8\epsilon (2^{m_1-2})^2+1.
\end{eqnarray*}
Identifying, we get
$$
(4m_1-k-3)2^{4m_1-k-6}-(4m_1-2k-2)(4m_1-2k-5) 2^{4m_1-2k-8}=\epsilon 2^{2m_1-1}-1.
$$
Note that $4m_1-2k-8$ is even. If $4m_1-2k-8\ge 0$, then the left--hand side is even and the right--hand side is odd, a contradiction. 
Thus, we must have $4m_1-2k-8=-2$. This gives $4m_1=2k+6$, so $m_1=(k+3)/2$. We thus get
$$
(k+3)2^{k}-1=\epsilon 2^{k+2}-1.
$$
This implies that $\epsilon=1$ and $(k+3) 2^{k}=2^{k+2}$, which leads to $k+3=4$, so $k=1$, which is impossible. Thus, this instance does not give us a solution either. 

Assume finally that $m_2\in [3k+4,4k+1]$. Applying the Cooper-Howard formula from Lemma \ref{teoHoward}, we get
$$
F_{m_2}^{(k)}=2^{m_2-2}+\sum_{j=1}^3 C_{m_2,j}2^{m_2-(k+1)j-2}.
$$
Eliminating the main term in the equality $F_{m_2}^{(k)}=x_4$ and changing signs in the remaining equation, we get
\begin{equation}
\label{eq:xxx}
\sum_{j=1}^3 -C_{m_2,j} 2^{m_2-(k+1)j-2}=\epsilon 2^{2m_1-1}-1.
\end{equation}
At $j=3$, the exponent of $2$ is $m_2-3j-5$. If this is positive,
the left hand side is even and the right--hand side is odd, a contradiction. Thus, $m_2\in \{3k+4,3k+5\}$. In this case,
$$
-C_{m_2,3} 2^{m_2-3k-5}=\left(\binom{m_2-3k}{3}-\binom{m_2-3k-2}{1}\right)2^{m_2-3k-5}\in \{1,7\}.
$$
For $j\in \{1,2\}$, $m_2-j(k+1)-2\ge m_2-2k-4\ge k>500$. Thus, the left--hand side in \eqref{eq:xxx} is congruent to $1,7\pmod {2^{500}}$, while the right--hand side 
of \eqref{eq:xxx} is congruent to $-1\pmod {2^{500}}$ because $m_1>500$. We thus get $1,7\equiv -1\pmod {2^{500}}$, a contradiction. Hence, there are no solutions with $n_2=4$.

\subsection{The case $n_2\ge 5$}

The goal here is to prove the following result.

\begin{lem}
If $k>500$ and $m_1>376$, then there is no solution with $n_2\ge 5$.
\end{lem}

We write again the two series for $2x_{n_2}=2F_{m_2}^{(k)}$:
$$
2F_{m_2}^{(k)}=2^{m_2-1}\left(1+\frac{k-m_2}{2^{k+1}}+\gamma_{m_2}\right)=2^{n_2(m_1-1)}\left(1+\frac{-\epsilon n_2}{2^{2(m_1-1)}}+\gamma_{n_2}'\right),
$$
where
$$
|\gamma_{m_2}|<\frac{1}{2^{4k/3}}\qquad {\text{\rm and}}\qquad |\gamma_{n_2}'|\le \sum_{i\ge 2} \frac{1}{2^{1.5 m_1 i}}\left(\frac{1}{2^i}\right)<\frac{1}{2^{3m_1}}.
$$
By Lemma \ref{n1=1}, we have $m_2-1=n_2(m_1-1)$ so the leading powers of $2$ above cancel, and we get
$$
\frac{k-m_2}{2^{k+1}}+\gamma_{m_2}=\frac{-\epsilon n_2}{2^{2(m_1-1)}}+\gamma_{n_2}'.
$$
We would like to derive that this implies that 
\begin{equation}
\label{eq:71}
\frac{k-m_2}{2^{k+1}}=\frac{-\epsilon n_2}{2^{2(m_1-1)}}.
\end{equation}
Well, we distinguish two cases. 

\medskip

{\bf Case 1.} {\it Suppose that $2(m_1-1)\ge k+1$.}

\medskip 

We then write
\begin{equation}
\label{eq:8}
\left|\frac{k-m_2}{2^{k+1}}+\frac{n_2\epsilon}{2^{2(m_1-1)}}\right|\le |\gamma_{m_2}|+|\gamma_{n_2}'|\le \frac{1}{2^{4k/3}}+\frac{1}{2^{3m_1}}.
\end{equation}
Since $2m_1\ge k+3$, we get $3m_1>3k/2>4k/3$. Thus,
\begin{equation}
\label{eq:x}
\left|\frac{k-m_2}{2^{k+1}}+\frac{n_2\epsilon}{2^{2(m_1-1)}}\right|\le \frac{2}{2^{4k/3}}.
\end{equation}
Suppose further that $m_1\le 2k/3$. Multiplying inequality \eqref{eq:x} across by $2^{2(m_1-1)}$, we get
$$
|2^{2(m_1-1)-(k+1)} (k-m_2)+\epsilon n_2|\le \frac{2^{2m_1-1}}{2^{4k/3}}\le \frac{1}{2},
$$
and since the left--hand side above is an integer, it must be the zero integer. This proves \eqref{eq:71} in the current case assuming that $m_1\le 2k/3$. If $m_1>2k/3$, we deduce from \eqref{eq:x}
that
$$
\frac{m_2-k}{2^{k+1}}<\frac{2}{2^{4k/3}}+\frac{n_2}{2^{2(m_1-1)}}<\frac{2+4n_2}{2^{4k/3}}<\frac{5n_2}{2^{4k/3}}<\frac{1}{2^{13k/12}},
$$
where in the right--above we used the fact that $8n_2<2^{k/4}$ (see \eqref{eq:16}). We thus get
$$
2\le m_2-k<\frac{2^{k+1}}{2^{13k/12}}<\frac{2}{2^{k/12}}<1,
$$
where the right--most inequality holds since $k>500$. This is a contradiction, so the $m_1>2k/3$ cannot occur in this case. This completes the proof of \eqref{eq:71} in Case 1.

\medskip

{\bf Case 2.} {\it Assume that $2(m_1-1)<k+1$.} 

\medskip

We then write
$$
\frac{n_2}{2^{2(m_1-1)}}\le \frac{m_2-k}{2^{k+1}}+|\gamma_{m_2}|+|\gamma_{n_2}'|.
$$
Since $|\gamma_{m_2}|<1/2^{4k/3}<1/2^{k+1}$ and $|\gamma_{n_2}'|\le 1/2^{3m_1}<1/2^{2(m_1-1)}$, we get that 
$$
\frac{1}{2^{2(m_1-1)}}<\left|\frac{n_2-1}{2^{2(m_1-1)}}\right|\le \frac{n_2}{2^{2(m_1-1)}}-|\gamma_{n_2}'|\le \frac{m_2-k}{2^{k+1}}+|\gamma_{m_2}|<\frac{m_2}{2^{k+1}}, 
$$
where we also used that $n_2>1$ and $k\ge 2$. Thus, 
$$
2^{k+1-2(m_1-1)}<m_2.
$$
We now go back to \eqref{eq:8} and write that
$$
\left|\frac{k-m_2}{2^{k+1}}+\frac{n_2\epsilon}{2^{2(m_1-1)}}\right|<\frac{2}{2^{\min\{4k/3,3m_1\}}}.
$$
We multiply across by $2^{k+1}$ getting
$$
|(k-m_2)+2^{k+1-2(m_1-1)}\epsilon n_2|<\frac{2^{k+2}}{2^{\min\{4k/3,3m_1\}}}.
$$
If the minimum on the right above is $4k/3$, then the right--hand side above is smaller than $4/2^{k/3}<1/2$ since $k$ is large, so the number on the left is zero. 
If the minimum is $3m_1$, on the right above then 
$$
|(k-m_2)+2^{k+1-2(m_1-1)}\epsilon n_2| < \frac{1}{2} \left(\frac{2^{k+1-2(m_1-1)}}{2^{m_1}}\right).
$$
Since 
$$
2^{k+1-2(m_1-1)}<m_2=n_2(m_1-1)<kn_2\le \max\{k^2,n_2^2\}<2^{m_1-6}<2^{m_1}
$$
(here, we used Lemma \ref{lem:7} for the inequality in the right--hand side above), it follows that 
$$
|(k-m_2)+2^{k+1-2(m_1-1)}\epsilon n_2| <\frac{1}{2},
$$
so again the left--hand side is $0$. Since $m_2>k$, this implies that $\epsilon=1$. We record what we just proved.

\begin{lem}
\label{lem:20}
If $k>500$, $m_1>376$ and $n_2\ge 5$, then $m_1\le k+1$, $n_1=1$, $\epsilon=1$, $m_2-1=n_2(m_1-1)$ and 
$$
\frac{m_2-k}{2^{k+1}}=\frac{n_2}{2^{2(m_1-1)}}.
$$
\end{lem}

We now get an extra relation. First, from Lemma \ref{lem:20}, we get that 
\begin{equation}
\label{eq:14}
n_2=\left\{\begin{matrix} 
2^{2(m_1-1)-(k+1)}(m_2-k) & {\text{\rm if}} & 2(m_1-1)\ge k+1;\\
\frac{m_2-k}{2^{k+1-2(m_1-1)}} & {\text{\rm if}} &  2(m_1-1)<k+1.\end{matrix}\right.
\end{equation}
Since $n_2\ge 5$, we can write more terms. 
\begin{eqnarray*}
2F_{m_2}^{(k)} & = & 2^{m_2-1}\left(1+\frac{k-m_2}{2^{k+1}}+\delta_{m_2}\frac{(m_2-2k+1)(m_2-2k-2)}{2^{2k+2}}+\eta_{m_2}\right)\\
2x_{n_2} & = & 2^{n_2(m_1-1)}\left(1+\frac{-\epsilon n_2}{2^{2(m_1-1)}}+\frac{n_2(n_2-3)}{2^{4(m_1-1)+1}}+\eta_{n_2}'\right)
\end{eqnarray*}
In the formula for $F_{m_2}^{(k)}$, we have $\delta_{m_2}=\zeta_{m_2}=0$ if $m_2\le 2k+2$. But $m_2\le 2k+2$ is not possible since then 
the only terms in the first expansion of $2F_{m_2}^{(k)}$ are the first two which already coincide with the first two terms 
of the expansion of $2x_{n_2}$, but in the second expansion we have additional terms since $n_2\ge 5$ while in the first we do not, which is a contradiction. Thus, $m_2\ge 2k+3$. 

\medskip

Assume that $2(m_1-1)\ge k+1$. In this case, from \eqref{eq:14}, we deduce that
$$
n_2=2^{2(m_1-1)-(k+1)} (m_2-k)=\frac{m_2-1}{m_1-1}.
$$
So, $m_2-k\mid m_2-1$. Thus, $m_2-k\mid (m_2-1)-(m_2-k)=k-1$. This shows that $m_2-k\le k-1$, so $m_2\le 2k-1$, a contradiction. Thus, $k+1>2(m_1-1)$.

\medskip

Simplifying again the power of $2$ from the two representations of $2x_{n_2}=2F_{m_2}^{(k)}$ and eliminating the first two terms we get
$$
\frac{(m_2-2k+1)(m_2-2k-2)}{2^{2k+3}}+\eta_{m_2}= \frac{n_2(n_2-3)}{2^{4(m_1-1)+1}}+\eta_{n_2}'.
$$
Here,
$$
|\eta_{m_2}|<\frac{4m_2^3}{2^{3k+3}}<\frac{1}{2^{2k+4}}\quad {\text{\rm and}}\quad  |\eta_{n_2}'|\le \sum_{i\ge 3} \frac{1}{2^{1.5 m i}}\left(\frac{1}{2^i}\right)<\frac{1}{2^{4.5 m_1+1}},
$$
by \eqref{eq:1} and \eqref{eq:7}. Thus, 
\begin{equation}
\label{eq:xx}
\left|\frac{(m_2-2k+1)(m_2-2k-2)}{2^{2k+3}}- \frac{n_2(n_2-3)}{2^{4(m_1-1)+1}}\right|\le |\eta_{m_2}|+|\eta_{n_2}'|<\frac{2}{\min\{ 2^{2k+4}, 2^{4.5m_1+1}\}}.
\end{equation}
Recall that $2(m_1-1)< k+1$. Then, by \eqref{eq:14}, we have $n_2\mid m_2-k$. Since also $n_2\mid m_2-1$, it follows that $n_2\mid (m_2-1)-(m_2-k)=k-1$. Thus, $n_2<k$, and since  
$2^{(k+1)-2(m_1-1)}$ is a divisor of $n_2$, we conclude that $2^{(k+1)-2(m_1-1)}<k$. We multiply \eqref{eq:xx} across by $2^{2(k+1)}$. We get
$$
\left| \frac{(m_2-2k+1)(m_2-2k-2)}{2}-2^{2(k+1)-4(m_1-1)}\frac{n_2(n_2-3)}{2}\right|\le \frac{2^{2k+3}}{\min\{2^{2k+4}, 2^{4.5m_1+1}\}}.
$$
If the minimum above is $2^{2k+4}$, then the right--hand side is $<\frac{1}{2}<1$. The left--hand side is an integer, so it equals $0$. If the minimum is 
$2^{4.5 m_1+1}$, then we can rewrite it as 
$$
\frac{2^{2k+3}}{2^{4.5m_1+1}}=\frac{2^{2(k+1)-4(m_1-1)}}{2^{0.5m_1+4}}<\frac{k^2}{2^{0.5m_1+5}}<1.
$$
The right--most inequality holds because $2^{m_1-6}>k^4$ by Lemma \ref{lem:7}. Hence, the left--hand side above is again $0$. We get that
\begin{equation}
\label{eq:21}
(m_2-2k+1)(m_2-2k-2)=2^{2(k+1)-4(m_1-1)} n_2(n_2-3).
\end{equation}
So, let us record the equations we have:
\begin{equation}
\label{eq:22}
\left\{ \begin{matrix} 
m_2-1 & = & n_2(m_1-1);\\
b & = & (k+1)-2(m_1-1);\\
n_2 & = & \frac{m_2-k}{2^{b}};\\
(m_2-2k+1)(m_2-2k-2) & = & 2^{2b} n_2(n_2-3).
\end{matrix}
\right.
\end{equation}
with $b>0$. To finish, we need to prove the following lemma.

\begin{lem}
\label{lem:last}
There are no integer solutions $(b,k,m_1,m_2,n_1,n_2)$ to system \eqref{eq:22} with $n_2\ge 5$ in the range $k>500$ and $m_1>376$.
\end{lem}

Now that we are seeing the light at the end of the tunnel, let's prove  Lemma \ref{lem:last}. As we saw, $n_2\mid (k-1)$. The last equation in system \eqref{eq:22} is 
$$
\left(\frac{m_2-1}{n_2}-\frac{2(k-1)}{n_2}\right)(m_2-2k-2)=2^{2b} (n_2-3),
$$
or, using the first equation in system \eqref{eq:22}, 
$$
\left(m_1-1-\frac{2(k-1)}{n_2}\right)(m_2-2k-2)=n_2-3.
$$
Now $n_2<k$ and $m_1\le k+1$, so from the first equation $m_2<k^2$. Since $2^{b}\mid m_2-k$, we get that $2^{b}<k^2$, so $b<2(\log k)/(\log 2)<3\log k$. Since $b=(k+1)-2(m_1-1)$, we get 
that 
$$
m_1=\frac{k+3-b}{2}\in \left(\frac{k+3-3\log k}{2}, \frac{k+3}{2}\right).
$$
In the last equation in the left, at most one of $m_1-1-2(k-1)/n_2$ (divisor of $m_2-2k+1$) and $m_2-2k-2$ is even. If the first one is even, then $m_2-2k-2$ is a divisor of $n_2-3$. Thus, 
$$
n_2-3\ge m_2-2k-2=n_2(m_1-1)-2k-1\ge n_2\left(\frac{k+1-3\log k}{2}\right)-2k-1,
$$
giving
$$
2k-2\ge n_2\left(\frac{k+1-3\log k}{2}-1\right)=n_2\left(\frac{k-1-3\log k}{2}\right).
$$
Since $n_2\ge 5$, we get
$$
4k-4\ge 5(k-1-3\log k),\qquad {\text{\rm or}}\qquad k\le 15\log k+1,
$$
giving $k\le 63$, a contradiction. Thus, $2^{2b}\mid m_2-2k-2$. Hence, 
$$
\left(m_1-1-\frac{2(k-1)}{n_2}\right)\left(\frac{m_2-2k-2}{2^{2b}}\right)=n_2-3,
$$
and all fractions above are in fact integers. The left--most integer is 
$$
m_1-1-\frac{2(k-1)}{n_2}\ge \frac{k+1-3\log k}{2}-\frac{2(k-1)}{5}>\frac{k-1}{12}-3
$$
since $k>500$. Since this number is a divisor of (so, at most as large as) the number $n_2-3=(k-1)/D-3$ for some integer $D$, we get that $D\in \{1,2,\ldots,11\}$. Thus, $(k-1)/D\in \{1,\ldots,11\}$, so 
$$
m_1-1-\frac{2(k-1)}{n_2}\ge \frac{k+1-3\log k}{2}-22=\frac{k-43-3\log k}{2}.
$$
Now let us look at the integer $(m_2-2k-2)/2^{2b}$. Assume that it is at least $3$. We then get
$$
3\left(\frac{k-43-3\log k}{2}\right)\le n_2-3\le k-4,\qquad {\text{\rm or}}\qquad k\le 121+9 \log k,
$$
and this is false for $k\ge 500$. Thus, $(m_2-2k-2)/2^{2b}\in \{1,2\}$. 

\noindent
Assume that $(m_2-2k-2)/2^{2b}=1$. Then 
$$
m_1-1-\frac{2(k-1)}{n_2}=n_2-3.
$$
The number in the left hand side is 
$$
m_1-1-\frac{2(k-1)}{n_2}\ge \frac{k+1-3\log k}{2}-22=\frac{k-43-3\log k}{2}>\frac{k-1}{3}-3
$$
(since $k>500$) and also
$$
m_1-1-\frac{2(k-1)}{n_2}\le m_1-3\le \frac{k-3}{2}<k-4.
$$
Thus, writing again $n_2=(k-1)/D$, we get that 
$$
n_1-3=\frac{k-1}{D}-3\in \left(\frac{k-1}{3}-3, \frac{k-1}{1}-3\right),
$$
showing that $1<D<3$, so $D=2$. Thus, $n_2=(k-1)/2$, and we get that
$$
\frac{k-7}{2}=\frac{k-1}{2}-3=n_2-3=m_1-1-\frac{2(k-1)}{n_2}=m_1-1-4=m_1-5,
$$
so 
$$
m_1=\frac{k+3}{2},\qquad {\text{\rm so}}\qquad b=0,
$$
which is impossible. 

Assume next that $(m_2-2k-2)/{2^{b}}=2$. In this case, we get
$$
n_2-3=2\left(m_1-1-\frac{2(k-1)}{n_2}\right).
$$
Proceeding as before, we have 
\begin{eqnarray*}
\frac{k-1}{D}-3 & = & n_2-3=2\left(m_1-1-\frac{2(k-1)}{n_2}\right)\ge 2\left(\frac{k+1-3\log k}{2}-22\right)\\
&  = &k-43-3\log k>\frac{k-1}{2}-3,
\end{eqnarray*}
showing that $D<2$. Thus, $D=1$ and so $n_2=k-1$. Hence,
$$
k-4=n_2-3=2\left(m_1-1-\frac{2(k-1)}{n_2}\right)=2(m_1-1-2)=2(m_1-3),
$$
so
$$ 
m_1=\frac{k+2}{2},\qquad {\text{\rm therefore}}\qquad b=1.
$$
Thus, $m_2-2k-2=2^{2b+1}=8$. Consequently, 
$$
8=(m_2-1)-2k-1=n_2(m_1-1)-2k-1=\frac{(k-1)k}{2}-2k-1=\frac{k^2-5k-2}{2},
$$
giving $k^2-5k-18=0$, which is impossible.

So, indeed there are no solutions with $k>500$ and $m_1>376$ other than the ones from (i) and (ii) of Theorem \ref{Main}. \qed

\section{The computational part $k\le 500$ or $m_1\le 376$}

Throughout this section, we make the following definition.

\begin{definition}
Assume that $k\ge 4,~x_1\ge 1,~\epsilon\in \{\pm 1\}$ are given such that there exist $n_1\ge 1$ and $m_1\ge 2$ such that $x_{n_1}=F_{m_1}^{(k)}$. 
We say that $n_1$ is minimal if there is are no  positive integers $n_0<n_1$ and $m_0<m_1$ such that the equality $x_{n_0}=F_{m_0}^{(k)}$ also holds.
\end{definition}

The aim of this section is to first show that in the range $k\le 500$ or $m_1\le 376$, all solutions of $x_{n_1}=F_{m_1}^{(k)}$ with $n_1$ minimal have $n_1=1$.   
Then we finish the calculations. 

\subsection{The case $k\le 500$}
\label{subcomp1}

Here, we exploit inequality \eqref{ineqqq2}, which we consider convenient to remind:
\begin{equation}
\label{ineqqq22}
|(n_2-n_1)\log(2f_k(\alpha)) -(n_1m_2-n_2m_1+n_2-n_1)\log\alpha|< \dfrac{6n_{2}}{\alpha^{m_1-1}}.
\end{equation}
Thus,
\begin{equation}
\label{eq:chi}
\left|\chi_k-\frac{N}{n_2-n_1}\right|<\frac{6n_2}{(n_2-n_1)\alpha^{m_1-1}\log \alpha},\qquad \chi_k:=\frac{\log(2f_k(\alpha))}{\log \alpha},
\end{equation}
with $N:=n_1m_2-n_2m_1+n_2-n_1$. Lemma \ref{lemma12} shows that
$$
n_2-n_1<n_2 < 8.2\times 10^{14}k^4(\log k)^3<10^{29}.
$$
The right--hand side of \eqref{eq:chi} can be rewritten as 
\begin{equation}
\label{eq:ass1}
\frac{1}{2(n_2-n_1)^2}\left(\frac{\alpha^{m_1-1}\log \alpha}{12n_2(n_2-n_1)}\right)^{-1}.
\end{equation}
Assume that
\begin{equation}
\label{eq:ass}
\frac{\alpha^{m_1-1}}{\log \alpha}>12 (8.2\times 10^{14} k^4 (\log k)^3)^2.
\end{equation}
Using $\alpha>1.927$, inequality \eqref{eq:ass} holds with $k\le 500$ for all $m_1\ge 203$. In this case, inequalities \eqref{eq:ass1}, \eqref{eq:chi} and Lemma 
\ref{lem:legendre} (i) show that $N/(n_1-n_1)=p_j^{(k)}/q_j^{(k)}$ for some $j\ge 0$, where $p_j^{(k)}/q_j^{(k)}$ is the $j$th convergent of $\chi_k$. Note that $\chi_k\in (0,1)$ because by Lemma \ref{fala5} (i), we have
$1<2f_k(\alpha)<1.5<\alpha$. 

\medskip

We distinguish two cases.

\medskip

{\bf Case 1.} $N\ne 0$.

\medskip

In this case, $j\ge 1$. Since 
$$
n_2-n_1\le 10^{29}<F_{150}\le q_{150}^{(k)},
$$
where $F_{150}$ is the $150$th member of the Fibonacci sequence, it follows that if we take
$$
Q:=\max\{a_i^{(k)}: 2\le i\le 150; 4\le k\le 500\},
$$
then Lemma \ref{lem:legendre} (ii) implies that 
$$
\frac{1}{(Q+2)(q_j^{(k)})^2}<\left|\chi_k-\frac{N}{n_2-n_1}\right|<\frac{6n_2}{(n_2-n_1)\alpha^{m_1-1}\log \alpha}.
$$
A computer calculation shows that $Q=433576$, so $Q+2<10^6$. Hence,
\begin{eqnarray*}
\alpha^{m_1-1}\log \alpha & < & 6n_2 (Q+2)(q_j^{(k)})^2(n_2-n_1)<6\times 10^6 n_2^2\\
& < & 6\times 10^6 (8.2\times 10^{14} 500^4 (\log 500)^3)^2,
\end{eqnarray*}
and using $\alpha\ge 1.927$, we get $m_1\le 221$. 

\medskip 

{\bf Case 2.} $N=0$.

\medskip

In this case, inequality \eqref{eq:chi} gives
$$
\alpha^{m_1-1}\log \alpha<6n_2\chi_k^{-1}<6\times (8.2\times 10^{14} 500^4(\log 500)^3) \chi_k^{-1}.
$$
A computation with {\it Mathematica} reveals that $\chi_k^{-1}<10^{148}$ for $k\le 500$. Feeding this into the above inequality, we get
$m_1\le 720$. Note that since $N=0$, we also have $n_1(m_2-1)=n_2(m_1-1)$. In particular, $n_1=1$ is not possible in this case.

\medskip

Let us record what we just proved.

\begin{lem}
\label{lem:23}
If $k\le 500$, then the following hold:
\begin{itemize}
\item[(i)] $m_1\le 221$;
\item[(ii)] $m_1\in [222,720]$, but $n_1>1$.
\end{itemize}
\end{lem}
For reasons that will become clear later, we allow $m\le 1049$ (instead of just $m\le 720$). 
To continue, assume first that $x_1\in \{1,2,3,\ldots,20\}$. We then generate all values of $\delta=x_1+{\sqrt{x_1^2-\epsilon}}$ for $\epsilon\in \{\pm 1\}$. We generate $x_{n_1}=(\delta^{n_1}+\eta^{n_1})/2$, where $\eta$ is the Galois conjugate of $\delta$ in the quadratic field ${\mathbb Q}(\delta)$, for all $1\le n\le m\le 1049$ and we test
for the equation
$$
x_{n}=F_{m}^{(k)}\qquad 4\le k\le 500,\quad 2\le m\le 1049.
$$
The only solutions we find  computationally have:
\begin{itemize}
\item[(i)] $n=1$ and $x_1\in \{1,2,4,8,15,16\}$;
\item[(ii)] $n=2$ and $x_2\in \{31,127,511\}$. These are not minimal because $x_2=31=F_7^{(5)}$ has $\epsilon=1$ and for it $x_1=4=F_7^{(4)}$, 
$x_2=127=F_9^{(7)}$ has $\epsilon=1$ and for it $x_1=8=F_9^{(5)}$, while $x_2=511=F_{11}^{(9)}$ has $\epsilon=1$ and 
for it $x_1=16=F_{11}^{(6)}$, as stated in (i) of Theorem \ref{Main} with $k=7$, $9$, and $11$, respectively. 
\item[(iii)] $n=3$ and $x_3=16336=F_{19}^{(13)}$. This is not minimal since $x_1=16=F_{6}^{(13)}$, as stated in (ii) of Theorem \ref{Main} with $a=1$. 
\end{itemize}
Assume now that $x_1\ge 21$. Then $\delta\ge 21+{\sqrt{440}}$. Inequality \eqref{eq:uuu} together with the fact that $m_1\le 1050$ gives
$$
n_1\le \frac{(m_1+1)\log \alpha}{\log \delta}\le \frac{1051\log 2}{\log(21+{\sqrt{440}})},
$$
so $n_1\le 194$. Our next goal is to show that in our range $k\le 500$ and $m\le 1049$, 
we must have $n\in \{1,2,3\}$. For this, assume that $n>3$. Every positive integer $>3$ is either divisible by $4,~6,~9$ or a prime $p\ge 5$. 
Thus, we generate the set 
$$
{\mathcal B}=\{4,6,9,p_k: 3\le k\le 44\},
$$
a set with $45$ elements, where $p_k$ is the $k$th prime. We use the fact that if $a\mid b$, then $x_b$ is the $a$th solution of the Pell equation whose first (smallest) $x$-coordinate is $x_{b/a}$ (that is, 
$\delta$ gets replaced by $\delta^{b/a}$). In particular, 
$x_{n_1}$ is $x_b$ for some $b\in B$ and some value of $x_1$. Further, say $y=F_m^{(k)}$ for some $m\in [2,1049]$ and $k\in [3,500]$. We then need to solve $x_b=y.$ Note that if $z\ge 1$ and $n\ge 2$, then 
\begin{equation}
\label{eq:z}
(z^n+1)^{1/n}-z=z\left(\left(1+\frac{1}{z^n}\right)^{1/n}-1\right)<\frac{1}{nz^{n-1}}\le \frac{1}{2}.
\end{equation}
Thus, 
$$
x_b=\left(x_1+{\sqrt{x_1^2-\epsilon}}\right)^b+\left(x_1-{\sqrt{x_1^2-\epsilon}}\right)^b=2y
$$
implies
$$
x_1+{\sqrt{x_1^2-1}}\in ((2y-1/2)^{1/b}, (2y+1/2)^{1/b}).
$$
Further, this leads to
$$
2x_1\in ((2y-1/2)^{1/n}-1/2, (2y+1/2)^{1/n}+1/2).
$$
The length of the interval on the right above is, by \eqref{eq:z}, at most $2$, so it contains at most one even integer $2x_1$ and if it contains one, it must be such that
\begin{equation}
\label{eq:x1}
x_1=\left\lfloor \frac{1}{2} \left(\left(2y+\frac{1}{2}\right)^{1/b}+\frac{1}{2}\right)\right\rfloor.
\end{equation}
So what we did was for each $y=F_m^{(k)}$ and each $b\in B$, we calculated the last $10$-digits of the integer shown at \eqref{eq:x1} (that is, we only calculated it modulo $10^{10}$). Then we picked 
$\epsilon\in \{\pm 1\}$ and generated $\{x_n\}_{n\ge 0}$ as the sequence given by $x_0:=1$, $x_1$ given by \eqref{eq:x1} modulo $10^{10}$ and $x_{n+1}=(2x_1) x_n-\epsilon x_{n-1} \pmod {10^{10}}$ for all $n\ge 1$. 
In this way, we never kept more that then last $10$ digits of $x_n$. And we checked whether indeed $x_b\equiv y \pmod {10^{10}}$. Unsurprisingly, no solution was found. 
We used the same program for $n_1=2,3$. For these we got that all solutions of (i) in our range were candidates for $n_1=2$ and all solutions (ii) in our range were candidates for $n_1=3$.
By candidates we meant that we only checked out these equalities modulo $10^{10}$. They turn out to be actual solutions
for $\epsilon=1$ (and they are not solutions with $\epsilon=-1$ just because a number of the form $2^{2j+1}-1$ with $j\ge 2$ cannot be also of the form $2z^2+1$ for some integer $z$, while a number of the form
$4x^3-3x$  for some integer $x>1$ then it cannot be also of the form $4z^3+3z$ for some integer $z$). Finally, one word about ``recognising" $y$ as  number of the form $F_m^{(k)}$. It follows from a result of 
Bravo and Luca \cite{BrLu} that the equation $F_m^{(k)}=F_n^{(\ell)}$ with $m\ge k+2,~n\ge \ell+2$ and $k>\ell\ge 4$ has no solutions $(m,k,n,\ell)$. Thus, if we already know a representation of a representation of $y$ as $F_m^{(k)}$ for some $m$ and $k\ge 4$, then it is unique. In particular, for $j\ge 2$, $F_{2j+3}^{(2j+1)}$  is the only representation of $2^{2j+1}-1$ as a $F_m^{(k)}$ for some positive integers $m$ and $k\ge 4$.

\subsection{The case $m_1\le 376$} We may assume that $k>500$, otherwise we are in the preceding case. Thus, $k>m_1$, so $n_1=1$. Thus, $\delta=2^{m_1-2}+{\sqrt{2^{2m_1-2}-\epsilon}}$
for all $m_1\ge 2$ and $\epsilon\in \{\pm 1\}$ (except for $m_1=2$, case in which only $\epsilon=1$ is possible). We now go back to the proof of Lemma \ref{lem:7} to get 
that the inequality \eqref{eq:41}, recalled below
\begin{equation}
\label{eq:60}
\left|n_2\log \delta-(m_2-1)\log 2\right|<\frac{1}{2^{2k/3-2}}
\end{equation}
implies \eqref{eq:boundLambda}, namely
$$
2k/3-2<195(\log \delta) \max\left\{10.5, \log(3m_2)\right\}^2.
$$
For us, $\log \delta\le m_1\log 2\le 376\log 2$. Using also the  upper bound from Lemma \ref{lemma12} on $m_2$, we get
$$
2k/3-2<195\times 376 (\log 2) \max\left\{10.5,\log(3\times 4.1\times 10^{22} k^7 (\log k)^6)\right\}^2,
$$
leading to $k<4\times 10^9$. Thus, by Lemma \ref{lemma12} again, 
$$
n_2< 8.2\times 10^{14} k^4(\log k)^3<8.2\times 10^{14} (4\times 10^9)^4 (\log(4\times 10^9))^3<10^{58}.
$$
Now \eqref{eq:60} gives
\begin{equation}
\label{eq:61}
\left|\frac{\log \delta}{\log 2}-\frac{m_2-1}{n_2}\right|<\frac{1}{(\log 2)2^{2k/3-1} n_2}.
\end{equation}
In our range, the right--hand side above is smaller than $1/(2n_2^2)$. Indeed, this is equivalent to $n_2<2^{2k/3-3} (\log 2)$, which holds provided that
$$
8.2\times 10^{14} k^4 (\log k)^3<2^{2k/3-3} (\log 2),
$$
which indeed holds for all $k>500$. Thus, $(m_2-1)/n_2=p_j/q_j$ is some convergent of $\log \delta/\log 2$. Since its denominator $q_j$ divides $n_2$ and 
$$
q_j\le n_2<10^{58}<F_{299},
$$
where $F_{299}$ is the $299$th term of the Fibonacci sequence, it follows that $j\le 298$. We generated the continued fractions of all $\log \delta/\log 2$ for all 
possibilities for $m_1\le 376$, $\epsilon\in \{\pm 1\}$ and $j\le 299$ and collected together the obtained values of $a_j$. The maximum value obtained  was $1033566$.
Hence, 
$$
\frac{1}{1.1\times 10^7 n_2^2}<\frac{1}{(a_{j+1}+2)n_2^2}<\left|\frac{\log \delta}{\log 2}-\frac{m_2-1}{n_2}\right|<\frac{1}{(\log 2)2^{2k/3-1} n_2},
$$
giving
$$
2^{2k/3-2}\log 2<1.1\times 10^7 n_2<1.1\times 10^7\times (8.2\times 10^{14} k^4 (\log k)^3),
$$
giving $k\le 166$, a contradiction. 

Thus, this case leads to no solution, and we must have $k\le 500$, $n_1=1$ and $m_1\le 221$ by Lemma \ref{lem:23}. 

\subsection{The final computations}

Now we go to inequality \eqref{eq:Gamma3} for $(n,m)=(n_2,m_2)$:
\begin{equation}
\label{eq:Gamma4}
|n_2\log\delta - \log(2f_{k}(\alpha))-(m_2-1)\log\alpha|<\frac{3}{\alpha^{m_2-1}}.
\end{equation}
We divide both sides by $\log\alpha$ and get
$$
|n_2\tau-(m_2-1)-\mu|<\frac{A}{B^{m_2-1}},\quad (\tau,\mu,A,B):=\left(\frac{\log \delta}{\log \alpha}, \frac{\log(2f_k(\alpha)}{\log \alpha},\frac{3}{\log(1.92)}, 1.92\right).
$$
We have 
$$
n_2\le 8.2\times 10^{14} k^4(\log k)^3\le 8.2 \times 10^{14} (500)^4 (\log 500)^3<1.3\times 10^{28}:=M.
$$
Since $6M<10^{30}<F_{150}$, we try $q_{\lambda}$ for some $\lambda\ge 150$. 
A computer code ran through the range $k\in [4,500]$, $m_1\in [2,221]$ and $\epsilon\in \{\pm 1\}$, generated 
$\delta=2^{m_1-2}+{\sqrt{2^{2(m_1-2)}-\epsilon}}$ (except for $m_1=2$, when only $\epsilon=1$ is possible), and confirmed the following:
\begin{itemize}
\item[(i)] For $4\le k\le 500$ and $\lambda=200$, we have $\varepsilon>0$ in all cases. 
\item[(ii)] The maximal value of $1+\lfloor \log(Aq_{\lambda}/\varepsilon)/\log B\rfloor$ in (i) above is $1049$.
\end{itemize}
Applying Lemma \ref{Dujjella}, we got that in all cases $m_2\le 1049$ by using $q_{200}$. By the calculations from Subsection \ref{subcomp1} where in fact 
we treated the case $m\le 1049$, we get that $(n_2,m_2)$ is one of the solutions listed in (i) or (ii) of Theorem \ref{Main}. This finishes the proof of the theorem.

\subsection*{Acknowledgements}
M.~D. was supported by the Austrian Science Fund (FWF) grants: F5510-N26 -- Part of the special research program (SFB), ``Quasi Monte Carlo Metods: Theory and Applications'', P26114-N26 --``Diophantine Problems: Analytic, geometric and computational aspects'' and W1230 --``Doctoral Program Discrete Mathematics''. Part of the work in this paper was done when both authors visited the Max Plank Institute for Mathematics Bonn, in March 2018. They thank this institution for hospitality and a fruitful working environment. F.~L. was also supported by grant CPRR160325161141 and an A-rated scientist award both from the NRF of South Africa and by grant no. 17-02804S of the Czech Granting Agency.

\end{document}